\documentclass[a4paper,12]{article}%
\usepackage[english]{babel}
\usepackage{amsmath}
\usepackage{graphicx}
\usepackage{epsfig,cite}
\usepackage{amsfonts}%
\usepackage{amssymb}
\usepackage{amsthm}

\usepackage{hhline}

\usepackage{colortbl}
\usepackage{xcolor}

\newtheorem{theorem}{Theorem}

\newtheorem{remark}{Remark}

\usepackage{algorithm2e}

\numberwithin{equation}{section}
\numberwithin{theorem}{section}
\numberwithin{corollary}{section}
\numberwithin{remark}{section}

\begin{document}
\markboth{V.L. MAKAROV, N.O.
ROSSOKHATA, D.V. DRAGUNOV}{Exponentially convergent functional-discrete method}
\title{\bf Exponentially convergent functional-discrete method for solving Sturm-Liouville problems with potential including Dirac $\delta$-function.}
\author{{\bf V.L. Makarov\footnote{E-mail: {\tt makarov@imath.kiev.ua}}, D.V. Dragunov\footnote{E-mail: {\tt dragunovdenis@gmail.com}}}\\[2mm]
Department of Numerical Mathematics, \\
Institute of  Mathematics of NAS of Ukraine, \\
3 Tereshchenkivs'ka Str., Kyiv-4, 01601, Ukraine\\[3mm]
{\bf N.O.Rossokhata\footnote{E-mail: {\tt nataross@gmail.com}}}\\[2mm]
Department of Mathematics and Statistics,\\
Concordia University,\\
1455 De Maisonneuve Blvd. West, QC H3G 1M8, Canada}
\maketitle
\begin{abstract}
In the paper we present a functional-discrete method for solving Sturm-Liouville problems with potential including function from $L_{1}(0,1)$ and $\delta$-function. For both, linear and nonlinear cases the sufficient conditions providing superexponential convergence rate of the method are obtained. The question of possible software implementation of the method is discussed in detail. The theoretical results are successfully confirmed by the numerical example included in the paper.
\end{abstract}


\section{Introduction}

The functional-discrete method (FD-method) was first  proposed for the Sturm-Liouville problem in \cite {makarov1}. The idea of this approach consists of embedding an original problem into a parametric set of problems with respect to a parameter $\tau$ in such a way that for $\tau=0$ we have a linear eigenvalue problem with piecewise constant coefficients and for $\tau=1$ we have the given eigenvalue problem. The transition from $\tau=0$ to $\tau=1$ using Taylor series results in a recursive algorithm. Thus, we represent the exact solution to the given problem, a pair (eigenvalue, eigenfunction), as two series, the first one --- for the eigenvalue, the other one --- for the eigenfunction. Then, an approximate solution is a pair of the corresponding truncated series. We find the first terms for both series via the {\it coefficient approximation method} (CAM) using piece-wise constant approximations. The CAM was first substantiated by Kryloff and Bogoliaubov (``tronson's'' method) in \cite{kryloff} and then developed in \cite{gordon, pruess, pryce, Ixaru}. As a result of piece-wise constant  approximation, we obtain an ``unperturbed'' problem which is also referenced to as the {\it basic problem}. The eigensolution to this problem gives us the first terms for the series representations mentioned above. Then, each successive term in the series for eigenfunction can be found as the exact solution to a linear boundary-value (not eigenvalue!) problem with a scalar parameter. A single value of this parameter which provides the solvability of the corresponding BVP gives us the next term in the series representation for eigenvalue. This approach allows us to find a numerical solution to the given eigenvalue problem with any desired accuracy, that is, we can improve the accuracy just by carrying out  a few more iterations.

This technique was developed for the Sturm-Liouville problems with continuous potential in
\cite{mrrew, bglm1,makarov2, mu1},
for the transmission eigenvalue problems with continuous potential in \cite{mrb1, mrb2, ross3.3}, for the linear eigenvalue problems with potential belonging to space $L_1$ in \cite{mr1, ross1}, for the nonlinear eigenvalue problems with continuous potential in \cite{ross4.1, ross4.2, makarov2008}. For all these cases the sufficient conditions providing the exponential convergence rate of the method have been found. It was also shown that the convergence improves along with the increase of the index of a trial eigenvalue.

In this article we extend the FD-approach to the case of a discontinuous potential which consists of the Delta-function and a function from space $L_1$. Such problems are of great interest in \cite{SavShkal_4,SavShkal_7, Albeverio_1, Albeverio_2}.

The paper is organized as follows. In section \ref{s_2} we state the rigorous formulation of the problem we deal with and briefly discuss several known properties of Sturm-Liouville problems with distribution potentials. In section \ref{s_3} the general description of the FD-method's algorithm is presented. The main theoretical results are stated and proved in section \ref{s_4}. Section \ref{s_5} is devoted to the question of the software implementation for the proposed method. Numerical example is given in section \ref{s_6}, followed by section \ref{s_7} containing  the conclusions.

\section{Formulation of the problem}\label{s_2}
In the paper we consider the following Sturm-Liouville problem\footnote{In this problem condition  $u^{\prime}(0)=1$ can be substituted by $\int\limits_{0}^{1}\bigl(u(x)\bigr)^{2}d x=const$ and the FD-method's algorithm presented below can be modified for treating problems of such type.}:
\begin{eqnarray}\label{eq_1}
    \frac{d^{2}}{d x^{2}}u(x)-\left[\beta\delta(x-\alpha)+q(x)\right]u(x)+\lambda u(x)-N(u(x))=0,\quad x\in (0,1)\backslash\{\alpha\}, \\
    u(0)=u(1)=0;\quad u^{\prime}(0)=1,\nonumber
\end{eqnarray}
where $\alpha\in (0,1),$ $\beta\geq 0,$  $q(x)\in L_{1}\left[0, 1\right],$ $N(u)=\sum\limits_{p=1}^{\infty}a_{p}u^{p}, \; \forall u\in\mathbb{R}.$ Here $\delta(x)$ denotes Dirac $\delta$-function or {\it impulse symbol,} which can be viewed as the derivative of the Heaviside step function (see \cite{Bracewell})
\begin{equation}\label{delta_function_properties}
    \frac{d}{d x}H(x)=\delta(x),\quad H(x)=
    \begin{cases}
    0, & x<0,\\
    \frac{1}{2}, & x=0,\\
    1, & x>0.
    \end{cases}
\end{equation}

Using \eqref{delta_function_properties} it is easy to verify that the solution $u(x)$ to problem \eqref{eq_1} satisfies the following equalities
\begin{eqnarray}\label{eq_2}
  u^{\prime}(x) &=& 1+\int\limits_{0}^{x}\left[q(x)u(x)-\lambda u(x)+N(u(x))\right]d x, \quad x \in \left[0, \alpha\right), \\
  u^{\prime}(x) &=& 1+\beta u(\alpha)+\int\limits_{0}^{x}\left[q(x)u(x)-\lambda u(x)+N(u(x))\right]d x,  \quad x \in \left(\alpha, 1\right]. \nonumber
\end{eqnarray}

Equalities \eqref{eq_2} imply that function $u^{\prime}(x)$ is discontinuous at the point $x=\alpha:$
\begin{equation}\label{eq_3}
    u^{\prime}(\alpha+0)-u^{\prime}(\alpha-0)=\beta u(\alpha);
\end{equation}

 Using formulas \eqref{eq_2} and \eqref{eq_3} we can conclude that problem \eqref{eq_1} is equivalent to the following one
\begin{eqnarray}\label{eq_4}
    \frac{d^{2}}{d x^{2}}u(x)-q(x)u(x)+\lambda u(x)-N(u(x))=0,\quad x\in(0,1)\backslash \{\alpha\}.\\
    u(0)=u(1)=0;\quad u^{\prime}(0)=1,\quad u^{\prime}(\alpha+0)-u^{\prime}(\alpha-0)=\beta u(\alpha).\nonumber
\end{eqnarray}
 Atkinson (see \cite{atkinson}, Ch.11) proved that linear Sturm-Liouville problem with a potential as a function of bounded variation preserves the following properties of the Sturm-Liouville problem with a continuous potential:
\begin{enumerate}
\item[({\bf a})]\label{assertion_a} all eigenvalues are real, simple, and form a monotone sequence increasing to infinity;
\item[({\bf b})]\label{assertion_b} the sequence of the corresponding normalized eigenfunctions forms a complete orthogonal system in $L_2(0,1)$.
\end{enumerate}
Also Atkinson have obtained asymptotic formulas for eigenvalues and corresponding eigenfunctions.

In \cite{vinokurov1, vinokurov2, vinokurov3, vinokurov4, vinokurov5, vinokurov6} the Sturm-Liouville problem with a potential as a function of bounded variation is studied using the generalized formulation of the problem. Particularly, in \cite{vinokurov3, vinokurov4, vinokurov5} it is shown that an eigenvalue $\lambda_n$ as a function of a potential $q$ is analytical, monotone, and linear with respect to constants.

\section{FD-method for solving Sturm-Liouville problem with potential including $\delta$-function} \label{s_3}
\subsection{General description of the method}

Let us consider the following generalization of problem   \eqref{eq_4}
\begin{eqnarray}
    & \frac{\partial^{2}}{\partial x^{2}}u(x, \tau)-\left[\tau q(x)+\lambda(\tau)\right] u(x, \tau)-\tau N(u(x,\tau))=0,\; x\in (0,1),\; x\neq\alpha \label{eq_5}  \\
    & u(0, \tau)=u(1, \tau)=0;\quad u^{\prime}(0,\tau)=1,\quad u(\alpha+0, \tau)-u(\alpha-0, \tau)=0,\label{eq_6}\\
    & \quad u^{\prime}_{x}(\alpha+0, \tau)- u^{\prime}_{x}(\alpha-0, \tau)=\beta u(\alpha, \tau),\quad \forall \tau\in \left[0, 1\right].\label{eq_7}
\end{eqnarray}

The fact that for $\tau=0$ we can easily find an exact solution $(\lambda(0), u_i(x,0), i=1,2)$ to problem \eqref{eq_5} -- \eqref{eq_7} and that for $\tau=1$ the solution $(\lambda(1), u_i(x,1), i=1,2)$ to problem \eqref{eq_5} -- \eqref{eq_7} coincides with the exact solution to  problem \eqref{eq_4}, suggests an idea to write the solution to problem \eqref{eq_5} -- \eqref{eq_7} in the form of series with respect to $\tau$
\begin{equation}\label{eq_8}
    u(x, \tau)=\sum\limits_{i=0}^{\infty}\stackrel{(i)}{u}\!\!\!(x)\tau^{i} \quad, \lambda(\tau)=\sum\limits_{i=0}^{\infty}\stackrel{(i)}{\lambda}\tau^{i}, \quad \forall x, \tau\in [0,1].
\end{equation}
To apply the FD-method's technique to the problem we also have to assume that
\begin{equation}\label{eq_9}
    \frac{\partial}{\partial x}u(x, \tau)=\sum\limits_{i=0}^{\infty}\frac{d}{d x}\stackrel{(i)}{u}\!\!\!(x)\tau^{i}, \quad \frac{\partial^{2}}{\partial x^{2}}u(x, \tau)=\sum\limits_{i=0}^{\infty}\frac{d^{2}}{d x^{2}}\stackrel{(i)}{u}\!\!\!(x)\tau^{i}
\end{equation}
for all $\tau \in [0,1]$ and for almost all $x\in(0,1).$

Setting $\tau=1$, we obtain the following representation
\begin{equation}\label{6}
\lambda=\lambda(1)=\sum_{j=0}^\infty\stackrel{(j)}{\lambda}, \hspace*{5mm} u(x)=u(x,1)=\sum_{j=0}^\infty \stackrel{(j)}{u}\!\!\!(x), \hspace*{5mm} i=1,2
\end{equation}
provided that these series converge.
Thus, we can represent the approximate solution to problem \eqref{eq_4} as the pair of corresponding truncated series
\begin{equation}\label{7}
\stackrel{m}{\lambda}=\sum_{j=0}^m\stackrel{(j)}{\lambda}, \hspace*{5mm} \stackrel{m}{u}\!\!\!(x)=\sum_{j=0}^m \stackrel{(j)}{u}\!\!\!(x), \hspace*{5mm} i=1,2,
\end{equation}
which is called the approximation of rank $m$.

Supposing that the function $u(x)\in C[0,1]$ \eqref{6} satisfies boundary conditions \eqref{eq_6} we arrive at the conclusion that the functions $\stackrel{(i)}{u}\!\!\!(x)$ are continuous on $[0, 1]$ and satisfy the conditions
\begin{eqnarray}\label{eq_10}
    \stackrel{(i)}{u}\!\!\!(0)=\stackrel{(i)}{u}\!\!\!(1)=0, \quad \forall i\in \mathbb{N}\cup \left\{0\right\},\nonumber \\ \left.\frac{d \stackrel{(0)}{u}\!\!\!(x)}{d x}\right|_{x=0}=1,\quad  \left.\frac{d \stackrel{(i)}{u}\!\!\!(x)}{d x}\right|_{x=0}=0,\quad \forall i\in \mathbb{N}.
\end{eqnarray}
To meet requirement \eqref{eq_7} we have to demand that
\begin{equation}\label{eq_11}
    \left.\frac{d\stackrel{(i)}{u}\!\!\!(x)}{d x}\right|_{x=\alpha+0}-\left.\frac{d\stackrel{(i)}{u}\!\!\!(x)}{d x}\right|_{x=\alpha-0}=\beta \stackrel{(i)}{u}\!\!\!(\alpha),\quad \forall i\in \mathbb{N}\cup \left\{0\right\}.
\end{equation}

Combining assumptions \eqref{eq_8}, \eqref{eq_9} together with equation \eqref{eq_5} we make the conclusion that unknown pairs $\bigl(\stackrel{(i)}{\lambda}, \stackrel{(i)}{u}(x)\bigr),$ $i\in \mathbb{N}\cup \{0\}$ can be found as the solutions to the following recurrence system of second-order differential equations
\begin{equation}\label{Recurrence_sequence_of_equations}
    \frac{d^{2}}{d x^{2}}\stackrel{(i)}{u}\!\!\!(x)+\stackrel{(0)}{\lambda}\stackrel{(i)}{u}\!\!\!(x)=\stackrel{(i)}{F}\!\!\!(x),\; x\in(0,1),\; x\neq \alpha,\quad i\in \mathbb{N}
\end{equation}
$$\stackrel{(i)}{F}\!\!\!(x)=-\sum\limits_{p=0}^{i-1}\stackrel{(i-p)}{\lambda}\stackrel{(p)}{u}\!\!\!(x)+q(x)\stackrel{(i-1)}{u}\!\!\!(x)+A_{i-1}\Bigl(N; \stackrel{(0)}{u}\!\!\!(x),\stackrel{(1)}{u}\!\!\!(x),\ldots, \stackrel{(i-1)}{u}\!\!\!(x)\Bigr)$$
\begin{equation}\label{Basic_equation}
    \frac{d^{2}}{d x^{2}}\stackrel{(0)}{u}\!\!\!(x)+\stackrel{(0)}{\lambda}\stackrel{(0)}{u}\!\!\!(x)=0,\; x\in(0,1),\; x\neq \alpha
\end{equation}
supplemented with conditions \eqref{eq_10}, \eqref{eq_11}, where
$A_{k}\left(N; v_{0}, v_{1}, \ldots, v_{k}\right)$ denotes the well-known Adomian polynomial (see \cite{seng1}, \cite{seng2}) of the order $k$ for nonlinear function $N(u)$
$$A_{k}\left(N; v_{0}, v_{1}, \ldots, v_{k}\right)=\left.\frac{d^{k}}{d t^{k}}N\left(\sum\limits_{p=0}^{\infty}v_{p}t^{p}\right)\right|_{t=0}.$$

\subsection{Basic problem}
Let us consider the problem of finding $\stackrel{(0)}{\lambda}$ and $\stackrel{(0)}{u}\!\!\!(x)$ in more detail. We call this problem {\it the basic problem}:
\begin{eqnarray}
  \frac{d^{2}}{d x^{2}}\stackrel{(0)}{u}\!\!\!(x)+\stackrel{(0)}{\lambda} \stackrel{(0)}{u}\!\!\!(x)=0,\; x\in(0,1),\; x\neq \alpha, && \label{eq_12}\\
  \stackrel{(0)}{u}\!\!\!(0)= \stackrel{(0)}{u}\!\!\!(1)=0, \quad \left.\frac{d}{d x}\stackrel{(0)}{u}\!\!\!(x)\right|_{x=0}=1,&&\label{eq_13}\\   \left.\frac{d\stackrel{(0)}{u}\!\!\!(x)}{d x}\right|_{x=\alpha+0}-\left.\frac{d\stackrel{(0)}{u}\!\!\!(x)}{d x}\right|_{x=\alpha-0}=\beta \stackrel{(0)}{u}\!\!\!(\alpha).\label{eq_14}
\end{eqnarray}

The unknown function $\stackrel{(0)}{u}\!\!\!(x)$ satisfying problem \eqref{eq_12}, \eqref{eq_13} can be represented in the following form
\begin{equation}\label{eq_15}
    \stackrel{(0)}{u}\!\!\!(x)=\stackrel{(0)}{u}\!\!\!(x,\stackrel{(0)}{\lambda})=\left\{
                 \begin{array}{cc}
                   \sin\left(\sqrt{\lambda^{(0)}}x\right)/\sqrt{\lambda^{(0)}} & x\in \left[0, \alpha\right], \\
                   \stackrel{(0)}{c}\sin\left(\sqrt{\lambda^{(0)}}(1-x)\right) & x\in \left(\alpha, 1\right], \\
                 \end{array}
               \right.
\end{equation}
where $\stackrel{(0)}{c}\in \mathbf{R}.$ Using condition \eqref{eq_14} and the fact that $\stackrel{(0)}{u}\!\!\!(x)$ is continuous on $[0,1]$ we obtain the following system of transcendental equations for determination of $\stackrel{(0)}{\lambda}$ and $\stackrel{(0)}{c}:$
\begin{eqnarray}
  \sin\Bigl(\sqrt{\stackrel{(0)}{\lambda}}\alpha\Bigr)/\sqrt{\stackrel{(0)}{\lambda}} = \stackrel{(0)}{c}\sin\Bigl(\sqrt{\stackrel{(0)}{\lambda}}(1-\alpha)\Bigr),\label{eq_16} \\
  -\stackrel{(0)}{c}\sqrt{\stackrel{(0)}{\lambda}}\cos\Bigl(\sqrt{\stackrel{(0)}{\lambda}}(1-\alpha)\Bigr)-\cos\Bigl(\sqrt{\stackrel{(0)}{\lambda}}\alpha\Bigr) = \beta \sin\Bigl(\sqrt{\stackrel{(0)}{\lambda}}\alpha\Bigr)/\sqrt{\stackrel{(0)}{\lambda}}.\label{eq_17}
\end{eqnarray}
Eliminating unknown parameter $\stackrel{(0)}{c}$ from system \eqref{eq_16}, \eqref{eq_17} we arrive at the equation with respect to $\stackrel{(0)}{\lambda}$
\begin{equation}\label{equation_for_lambda^{(0)}}
    \sqrt{\stackrel{(0)}{\lambda}}\sin\Bigl(\sqrt{\stackrel{(0)}{\lambda}}\Bigr)=-\beta\sin\Bigl(\sqrt{\stackrel{(0)}{\lambda}}\alpha\Bigr)\sin\Bigl(\sqrt{\stackrel{(0)}{\lambda}}(1-\alpha)\Bigr)
\end{equation}
and at the following expression for $\stackrel{(0)}{c}$
\begin{equation}\label{expression_for_c^{(0)}}
   \stackrel{(0)}{c}=\stackrel{(0)}{c}(\stackrel{(0)}{\lambda})=\left\{
                          \begin{array}{cc}
                            \frac{\sin\Bigl(\sqrt{\stackrel{(0)}{\lambda}}\alpha\Bigr)}{\sqrt{\stackrel{(0)}{\lambda}}\sin\Bigl(\sqrt{\stackrel{(0)}{\lambda}}(1-\alpha)\Bigr)},& \sin\Bigl(\sqrt{\stackrel{(0)}{\lambda}}(1-\alpha)\Bigr)\neq 0, \\
                            -\frac{\cos\Bigl(\sqrt{\stackrel{(0)}{\lambda}}\Bigr)}{\sqrt{\stackrel{(0)}{\lambda}}} ,& \sin\Bigl(\sqrt{\stackrel{(0)}{\lambda}}(1-\alpha)\Bigr)= 0.\\
                          \end{array}
                        \right.
\end{equation}

\begin{theorem}\label{basic_teorem}
Suppose that $\alpha\in(0,1)$ and $k\in \mathbf{N},$ then
 the interval $[\pi^{2} k^{2}, \pi^{2} (k+1)^{2})$ contains precisely one root of  equation \eqref{equation_for_lambda^{(0)}}, provided that $\beta\geq 0$.
\end{theorem}
\begin{proof}
Suppose that $\beta\geq 0$ and fix some positive integer $k.$

First of all we consider the case of $\beta=0.$ It is easy to check that in this case equation \eqref{equation_for_lambda^{(0)}} possesses  the countable set of solutions \begin{equation}\label{Solutions_beta=0}
   \stackrel{(0)}{\lambda}=\pi^{2}k^{2},\;\;\;\forall k\in \mathbf{N}.
\end{equation}
There are no other solutions of equation \eqref{equation_for_lambda^{(0)}} except those presented in \eqref{Solutions_beta=0}\footnote{We do not take into account solutions with $\stackrel{(0)}{\lambda}\leq 0.$}. We have that assertion of the theorem holds evidently.

From now on we assume that $\beta> 0.$ For the given fixed $k\in\mathbf{N}$ there can be only two possibilities:
1) $\sin(\pi k\alpha)=0;$
2) $\sin(\pi k\alpha)\neq 0.$

Let us consider case 1). We have that in this case the value $\stackrel{(0)}{\lambda}=\pi^{2} k^{2}$ satisfies equation \eqref{equation_for_lambda^{(0)}}. Hence, to prove the theorem  it remains only to show that the interval $(\pi^{2} k^{2}, \pi^{2} (k+1)^{2})$ contains no other roots of equation \eqref{equation_for_lambda^{(0)}}.    For this purpose we need to modify equation \eqref{equation_for_lambda^{(0)}} like the following
 \begin{equation}\label{modified equation}
    -\frac{y}{\beta}=f(y),\quad y\in (\pi k, \pi (k+1)),
 \end{equation}
 where
 $$ f(y)=\frac{\sin\Bigl(y\alpha\Bigr)\sin\Bigl(y(1-\alpha)\Bigr)}{\sin\Bigl(y\Bigr)},\quad y=\sqrt{\stackrel{(0)}{\lambda}}.$$

 It is not hard to verify that
\begin{equation}\label{df}
    f^{\prime}(y)=\frac{\alpha\sin^{2}(y(1-\alpha))+(1-\alpha)\sin^{2}(y\alpha)}{\sin^{2}(y)}
\end{equation}
and $\lim\limits_{y\rightarrow \pi k+0}f(y)=0.$
Equality $\sin(\pi k\alpha)=0$ implies that
$\sin(\pi (k+1)\alpha)\neq0$ and we have that
$\lim\limits_{y\rightarrow \pi (k+1)-0}f(y)=+\infty.$
Taking into account the positiveness of $f^{\prime}(y)$ on $(\pi k, \pi (k+1))$ (see \eqref{df}) we arrive at the conclusion that $(\pi k, \pi (k+1))\stackrel{f}{\rightarrow}(0, +\infty).$ This means that the graph of function $z=-\frac{y}{\beta}$ can't intersect the graph of function $z=f(y)$ on the interval $(\pi k, \pi (k+1))$ (see Fig. \ref{pic_1_basicpr}) and equation \eqref{equation_for_lambda^{(0)}} has no roots on the interval $(\pi^{2}k^{2}, \pi^{2}(k+1)^{2}).$ For case 1) the theorem is proved.

\begin{figure}[h!]
\begin{minipage}[h]{1\linewidth}
\begin{minipage}[h]{0.48\linewidth}
\center{\rotatebox{-0}{\includegraphics[
width=1.0\linewidth]{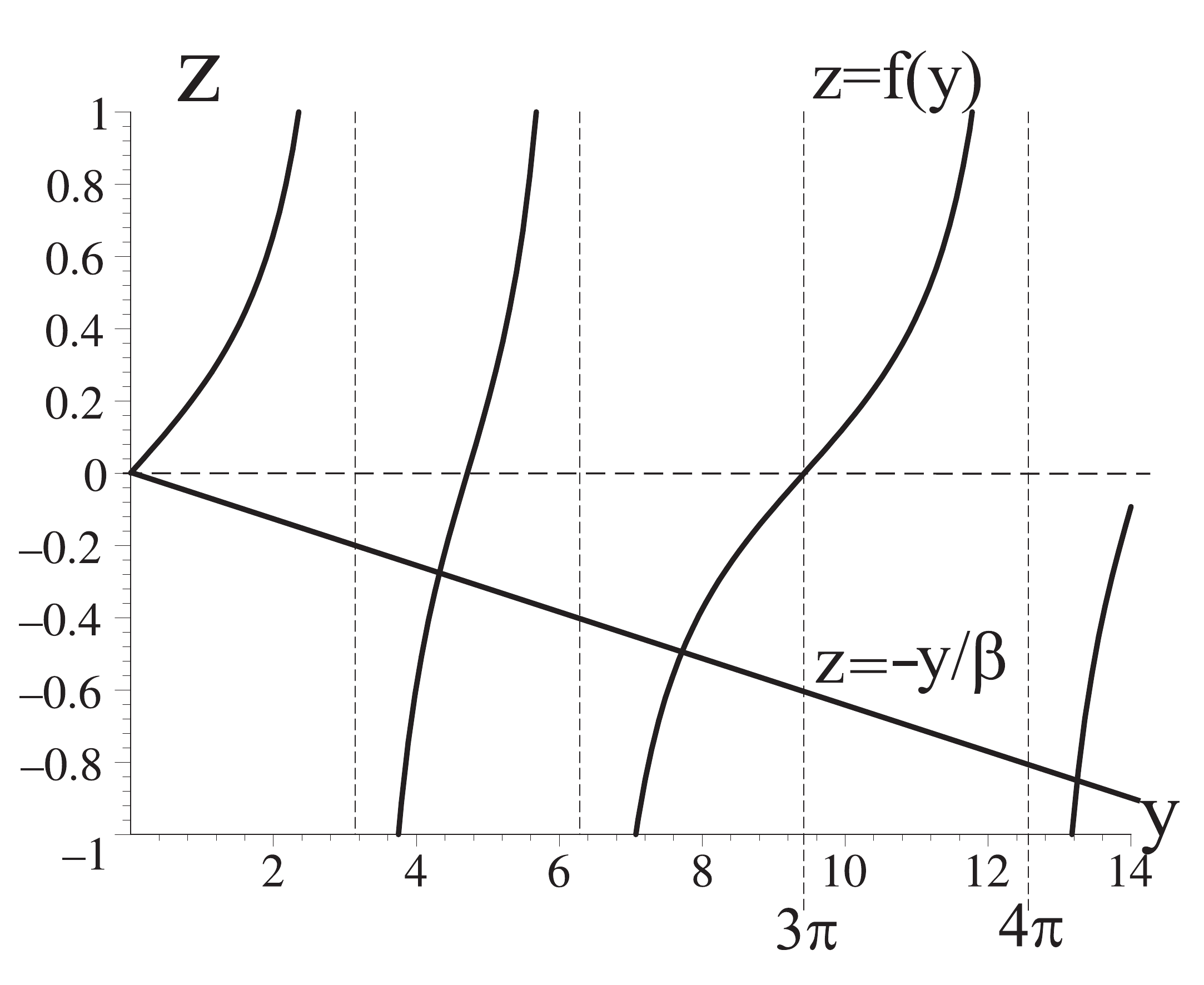}} \\ a)}
\caption{The graphs of the functions $z=f(x)$ \eqref{modified equation} and $z=-y/\beta,$ with $\alpha=1/3,$ $\beta=15.$ Equation \eqref{modified equation} possesses no roots on the interval $(\pi k, \pi (k+1))$ with $k=3.$
}\label{pic_1_basicpr}
\end{minipage}
\hfill
\begin{minipage}[h]{0.48\linewidth}
\center{\rotatebox{-0}{\includegraphics[
width=1.0\linewidth]{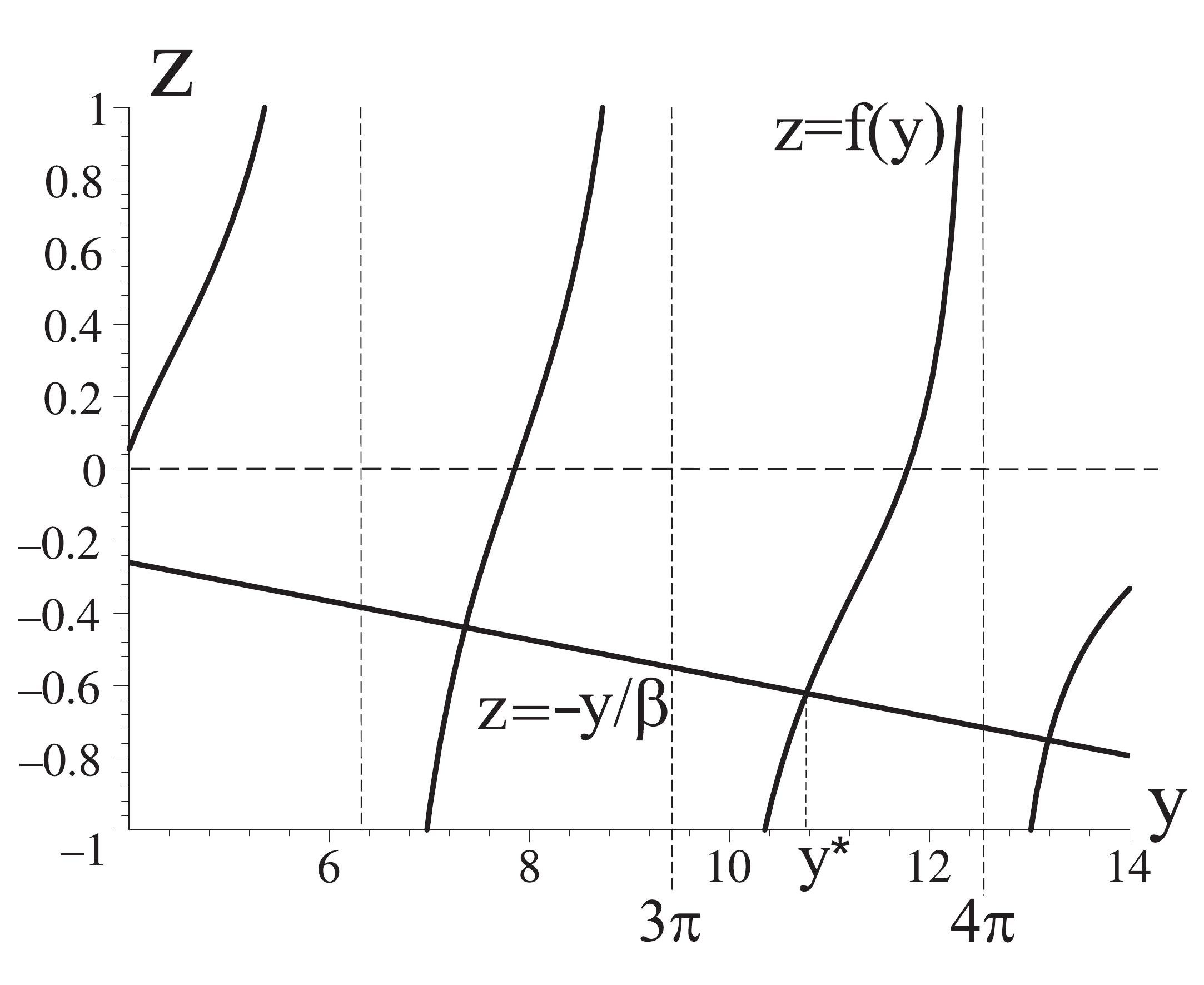}} \\ b)}
\caption{The graphs of the functions $z=f(x)$ \eqref{modified equation} and $z=-y/\beta,$ with $\alpha=1/5,$ $\beta=15.$ Here $y^{\ast}$ denotes the solution to equation \eqref{modified equation} on the interval $(\pi k, \pi (k+1))$ with $k=3.$
}\label{pic_2_basicpr}
\end{minipage}
\end{minipage}
\end{figure}

 It is worth to emphasize that case 1) is possible if and only if $\alpha=\frac{l}{k},$ where $l\in \mathbf{N},$ $l<k.$

And now we switch our attention to case 2) $\sin(\pi k\alpha)\neq 0.$  It is easy to see that in this case the value $\stackrel{(0)}{\lambda}=\pi^{2} k^{2}$ do not satisfies equation \eqref{equation_for_lambda^{(0)}}. We intend to show that equation \eqref{equation_for_lambda^{(0)}} possesses precisely one root on the interval $(\pi^{2}k^{2}, \pi^{2}(k+1)^{2}).$ For this purpose we will use equation \eqref{modified equation} again.
As it was established earlier (see formula \eqref{df}) the function $z=f(y)$ is strictly increasing on the interval $(\pi k, \pi (k+1)).$ Furthermore, it is easy to check that $$\lim\limits_{y\rightarrow \pi k+0}f(y)=-\infty,\quad \lim\limits_{y\rightarrow \pi (k+1)-0}f(y)=
\left\{
\begin{array}{cc}
+\infty, & \mbox{if}\quad \sin(\pi (k+1)\alpha)\neq 0,\\
0,& \mbox{if}\quad \sin(\pi (k+1)\alpha)= 0  \\
\end{array}
\right.
$$
and we have $(\pi k, \pi (k+1))\stackrel{f}{\rightarrow}(-\infty, 0).$ The last fact provides that the graph of function $z=f(y)$ necessarily intersects with the graph of function $z=-\frac{y}{\beta}$ on the interval $(\pi k, \pi (k+1))$  in a single point $y=y^{\ast},$ see Fig. \ref{pic_2_basicpr}. This point is a unique solution to equation \eqref{modified equation} on $(\pi k, \pi (k+1))$ and $\stackrel{(0)}{\lambda}=\left(y^{\ast}\right)^{2}$ satisfies equation  \eqref{equation_for_lambda^{(0)}}. Hence, assertion of the theorem holds true, which was to be proved.

The proof is completed.
\end{proof}

Proving Theorem \ref{basic_teorem} we have established an interesting property of the basic problem's solutions in regard to parameter $\alpha.$ We can formulate this property as the following remark.
\begin{remark}
   If $\alpha=\dfrac{m}{n},$ $m,n\in \mathbf{N},$ $m<n,$ then $\stackrel{(0)}{\lambda}= \pi k n$ satisfies equation \eqref{equation_for_lambda^{(0)}} for all $k\in \mathbf{N}.$
\end{remark}

Theorem \ref{basic_teorem} allows us to enumerate the roots of equation \eqref{equation_for_lambda^{(0)}}. We will denote by $\stackrel{(0)}{\lambda_{n}}$ the root of equation \eqref{equation_for_lambda^{(0)}} lied in the interval
$[\pi^{2} k^{2}, \pi^{2} (k+1)^{2}).$ We also will use the notation (see \eqref{eq_15}, \eqref{expression_for_c^{(0)}})
$$\stackrel{(0)}{u_{n}}(x)=\stackrel{(0)}{u}\!\!\!(x, \stackrel{(0)}{\lambda_{n}}), \quad \stackrel{(0)}{c_{n}}=\stackrel{(0)}{c}\!\!\!(\stackrel{(0)}{\lambda_{u}}).$$

Before passing to the next section we should make the following remark.
\begin{remark}\label{rem_3}

  The value $\stackrel{(0)}{c_{n}},$ $n\in\mathbf{N}$ is always nonzero and its value squared can be represented  in the following form
  \begin{equation}\label{c_square}
    \Bigl(\stackrel{(0)}{c_{n}}\Bigr)^{2}=\frac{1}{\stackrel{(0)}{\lambda_{n}}}\left[\sin^{2}\Bigl(\sqrt{\stackrel{(0)}{\lambda_{n}}}\alpha\Bigr)+\left(\cos\Bigl(\sqrt{\stackrel{(0)}{\lambda_{n}}}\alpha\Bigr)+\frac{\beta}{\sqrt{\stackrel{(0)}{\lambda_{n}}}}\sin\Bigl(\sqrt{\stackrel{(0)}{\lambda_{n}}}\alpha\Bigr)\right)^{2}\right]=\frac{\widetilde{c}_{n}}{\stackrel{(0)}{\lambda_{n}}}.
  \end{equation}

\end{remark}

\section{Convergence of the FD-method}\label{s_4}
In this section we will investigate the question of sufficient conditions providing the convergence of series \eqref{6}. To obtain such conditions we will use the method of generating functions and the main part of this section is devoted to the derivation of the appropriate estimates for the terms of series \eqref{6}.

\subsection{Equation for the generating function}

Suppose that $\stackrel{(0)}{\lambda_{n}}>0$ is an arbitrary eigenvalue of the Sturm-Liouville problem \eqref{eq_12}--\eqref{eq_14}.
The pairs $\stackrel{(j)}{\lambda_{n}},$ $\stackrel{(j)}{u_{n}}\!\!(x)$ for $j=1\ldots m$ can be found as the solutions to the second order differential equations (see \eqref{Recurrence_sequence_of_equations})
\begin{equation}\label{n_Recurrence_sequence_of_equations}
    \frac{d^{2}}{d x^{2}}\stackrel{(j)}{u_{n}}\!\!(x)+\stackrel{(0)}{\lambda_{n}}\stackrel{(j)}{u_{n}}\!\!(x)=\stackrel{(j)}{F_{n}}\!\!(x),\; x\in[0,1],\; x\neq\alpha,
\end{equation}
$$\stackrel{(j)}{F_{n}}\!\!(x)=-\sum\limits_{p=0}^{j-1}\stackrel{(j-p)}{\lambda_{n}}\stackrel{(p)}{u_{n}}\!\!(x)+q(x)\stackrel{(j-1)}{u_{n}}\!\!(x)+A_{j-1}\Bigl(N; \stackrel{(0)}{u_{n}}\!\!(x),\stackrel{(1)}{u_{n}}\!\!(x),\ldots, \stackrel{(j-1)}{u_{n}}\!\!(x)\Bigr)$$
with boundary conditions  ( see \eqref{eq_10})
\begin{eqnarray}\label{n_eq_10}
    \stackrel{(j)}{u_{n}}\!\!(0)=\stackrel{(j)}{u_{n}}\!\!(1)=0,\quad  \left.\frac{d \stackrel{(j)}{u_{n}}\!\!(x)}{d x}\right|_{x=0}=0,
\end{eqnarray}
and matching  conditions ( see \eqref{eq_11})
\begin{equation}\label{n_eq_11}
    \left.\frac{d\stackrel{(i)}{u_{n}}\!\!(x)}{d x}\right|_{x=\alpha+0}-\left.\frac{d\stackrel{(i)}{u_{n}}\!\!(x)}{d x}\right|_{x=\alpha-0}=\beta \stackrel{(i)}{u_{n}}(\alpha).
\end{equation}

The general solution to the $j$-th equation of system \eqref{n_Recurrence_sequence_of_equations} can be represented in the following form
\begin{equation}\label{u_j_left_representation}
    \stackrel{(j)}{u_{n}}\!\!(x)=\int\limits_{0}^{x}\frac{\sin\Bigl(\sqrt{\stackrel{(0)}{\lambda_{n}}}(x-\xi)\Bigr)}{\sqrt{\stackrel{(0)}{\lambda_{n}}}}\stackrel{(j)}{F_{n}}\!\!(\xi)d\xi, \quad x\in \left[0,\alpha\right],
\end{equation}
\begin{equation}\label{u_j_right_representation}
    \stackrel{(j)}{u_{n}}\!\!(x)=\stackrel{(j)}{c_{n}}\sin\Bigl(\sqrt{\stackrel{(0)}{\lambda_n}}(1-x)\Bigr)- \int\limits_{x}^{1}\frac{\sin\Bigl(\sqrt{\stackrel{(0)}{\lambda_n}}(x-\xi)\Bigr)}{\sqrt{\stackrel{(0)}{\lambda_n}}}\stackrel{(j)}{F}_{n}\!\!(\xi)d\xi, \quad x\in \left[\alpha,1\right].
\end{equation}
For equation \eqref{n_Recurrence_sequence_of_equations} to be solvable, the orthogonality condition
\begin{equation}\label{ortho_condition}
  \int\limits_{0}^{1}\stackrel{(j)}{F_{n}}(\xi)\stackrel{(0)}{u_{n}}(\xi)d\xi=0
\end{equation}
have to be provided. Condition \eqref{ortho_condition} allows us to find $\stackrel{(j)}{\lambda}_{n}:$
\begin{equation}\label{lambda_j}
    \stackrel{(j)}{\lambda_{n}}=\frac{\stackrel{(0)}{\lambda_{n}}}{M_{n}}\left(-\sum\limits_{p=1}^{j-1}\stackrel{(j-p)}{\lambda_{n}}\int\limits_{0}^{1}\stackrel{(0)}{u_{n}}(\xi)\stackrel{(p)}{u_{n}}(\xi)d\xi+\right.
\end{equation}
$$\left. +\int\limits_{0}^{1}q(\xi)\stackrel{(j-1)}{u_{n}}(\xi)\stackrel{(0)}{u_{n}}(\xi)d\xi+\int\limits_{0}^{1}A_{j-1}\Bigl(N; \stackrel{(0)}{u_{n}}(x),\stackrel{(1)}{u_{n}}(x),\ldots, \stackrel{(j-1)}{u_{n}}(x)\Bigr)\stackrel{(0)}{u_{n}}(\xi)d\xi\right),$$
where
$$   \frac{M_{n}}{\stackrel{(0)}{\lambda_{n}}}=\int\limits_{0}^{1}\Bigl(\stackrel{(0)}{u_{n}}(\xi)\Bigr)^{2}d\xi=\int\limits_{0}^{\alpha}\frac{\sin^{2}\Bigl(\sqrt{\stackrel{(0)}{\lambda_{n}}}x\Bigr)}{\stackrel{(0)}{\lambda_{n}}}d x+\Bigl(\stackrel{(0)}{c_{n}}\Bigr)^{2}\int\limits_{0}^{\alpha}\sin^{2}\Bigl(\sqrt{\stackrel{(0)}{\lambda_{n}}}(1-x)\Bigr)d x=$$
$$=\frac{1}{\stackrel{(0)}{\lambda_{n}}}\biggl(\Biggl.\Bigl(\frac{x}{2}-\frac{\sin\Bigl(2\sqrt{\stackrel{(0)}{\lambda_{n}}}x\Bigr)}{4\sqrt{\stackrel{(0)}{\lambda_{n}}}}\Bigr)\Biggr|_{0}^{\alpha}\biggr)+\Bigl(\stackrel{(0)}{c_{n}}\Bigr)^{2}\Biggl.\Biggl(\frac{x}{2}+\frac{\sin\Bigl(2\sqrt{\stackrel{(0)}{\lambda_{n}}}(1-x)\Bigr)}{4\sqrt{\stackrel{(0)}{\lambda_{n}}}}\Biggr)\Biggr|_{\alpha}^{1}=$$
$$=\frac{1}{2\stackrel{(0)}{\lambda_{n}}}\Biggl(\alpha-\frac{\sin\Bigl(2\sqrt{\stackrel{(0)}{\lambda_{n}}}\alpha\Bigr)}{2\sqrt{\stackrel{(0)}{\lambda_{n}}}}\Biggr)+\frac{\Bigl(\stackrel{(0)}{c_{n}}\Bigr)^{2}}{2}\Biggl((1-\alpha)-\frac{\sin\Bigl(2\sqrt{\stackrel{(0)}{\lambda_{n}}}(1-\alpha)\Bigr)}{2\sqrt{\stackrel{(0)}{\lambda_{n}}}}\Biggr).$$
It is easy to verify that system  \eqref{eq_16}, \eqref{eq_17} implies the equality
$$\Bigl(\stackrel{(0)}{c_{n}}\Bigr)^{2}\frac{\sin\Bigl(2\sqrt{\stackrel{(0)}{\lambda_{n}}}(1-\alpha)\Bigr)}{2\sqrt{\stackrel{(0)}{\lambda_{n}}}}=-\frac{\sin\Bigl(2\sqrt{\stackrel{(0)}{\lambda_{n}}}(\alpha)\Bigr)}{2\stackrel{(0)}{\lambda_{n}}\sqrt{\stackrel{(0)}{\lambda_{n}}}}-\frac{\beta\sin^{2}\Bigl(\sqrt{\stackrel{(0)}{\lambda_{n}}}\alpha\Bigr)}{\Bigl(\stackrel{(0)}{\lambda_{n}}\Bigr)^{2}},$$
using which we can obtain the following representation for $M_{n}:$
\begin{equation}\label{F_1}
M_{n}=\frac{\stackrel{(0)}{\lambda_{n}}}{2}\Biggl(\alpha/\stackrel{(0)}{\lambda_{n}}+\Bigl(\stackrel{(0)}{c_{n}}\Bigr)^{2}(1-\alpha)+\beta\sin^{2}\Bigl(\sqrt{\stackrel{(0)}{\lambda_{n}}}\alpha\Bigr)/\Bigl(\stackrel{(0)}{\lambda_{n}}\Bigr)^{2}\Biggr).
\end{equation}
On the other hand, formula \eqref{c_square} can be simplified like the following
\begin{equation}\label{F_2}
\Bigl(\stackrel{(0)}{c_{n}}\Bigr)^{2}=\frac{1}{\stackrel{(0)}{\lambda_{n}}}\left[1+\frac{2\beta}{\sqrt{\stackrel{(0)}{\lambda_{n}}}}\cos\Bigl(\sqrt{\stackrel{(0)}{\lambda_{n}}}\alpha\Bigr)\sin\Bigl(\sqrt{\stackrel{(0)}{\lambda_{n}}}\alpha\Bigr)+\frac{\beta^{2}}{\stackrel{(0)}{\lambda_{n}}}\sin^{2}\Bigl(\sqrt{\stackrel{(0)}{\lambda_{n}}}\alpha\Bigr)\right].
\end{equation}

Combining formulas \eqref{F_1} and \eqref{F_2} we arrive at the following expression for $M_{n}$
\begin{equation}\label{M_n}
    M_{n}=\frac{1}{2}\Biggl(1+\Biggl[\beta\sin\Bigl(2\sqrt{\stackrel{(0)}{\lambda_{n}}}\alpha\Bigr)/\sqrt{\stackrel{(0)}{\lambda_{n}}}+\beta^{2}\sin^{2}\Bigl(\sqrt{\stackrel{(0)}{\lambda_{n}}}\alpha\Bigr)/\stackrel{(0)}{\lambda_{n}}\Biggr](1-\alpha)+\Biggr.
\end{equation}
$$\Biggl.+\beta\sin^{2}\Bigl(\sqrt{\stackrel{(0)}{\lambda_{n}}}\alpha\Bigr)/\stackrel{(0)}{\lambda_{n}}\Biggr).$$

For the function $\stackrel{(j)}{u_{n}}\!\!\!(x)$ \eqref{u_j_left_representation}, \eqref{u_j_right_representation} to be continuous in the interval $[0,1]$ and satisfy condition \eqref{n_eq_11}, the  parameter $\stackrel{(j)}{c_{n}}$ in formula \eqref{u_j_right_representation} should satisfy the following equations
\begin{equation}\label{c_j_formula_1}
    \stackrel{(j)}{c_{n}}\sin\Bigl(\sqrt{\stackrel{(0)}{\lambda_{n}}}(1-\alpha)\Bigr)=\int\limits_{0}^{1}\frac{\sin\Bigl(\sqrt{\stackrel{(0)}{\lambda_{n}}}(\alpha-\xi)\Bigr)}{\sqrt{\stackrel{(0)}{\lambda_{n}}}}\stackrel{(j)}{F_{n}}\!\!\!(\xi)d\xi
\end{equation}
\begin{equation}\label{c_j_formula_2}
    \stackrel{(j)}{c_{n}}\cos\Bigl(\sqrt{\stackrel{(0)}{\lambda_{n}}}(1-\alpha)\Bigr)=-\int\limits_{0}^{1}\frac{\cos\Bigl(\sqrt{\stackrel{(0)}{\lambda_{n}}}(\alpha-\xi)\Bigr)}{\sqrt{\stackrel{(0)}{\lambda_{n}}}}\stackrel{(j)}{F_{n}}\!\!\!(\xi)d\xi-\beta
    \int\limits_{0}^{\alpha}\frac{\sin\Bigl(\sqrt{\stackrel{(0)}{\lambda_{n}}}(\alpha-\xi)\Bigr)}{\stackrel{(0)}{\lambda_{n}}}\stackrel{(j)}{F_{n}}\!\!\!(\xi)d\xi.
\end{equation}

Formulas \eqref{c_j_formula_1} and \eqref{c_j_formula_2} will be used in the software algorithm of the FD-method, as it is described in section \ref{s_5}. However, to obtain the convenient estimate for $\bigl|\stackrel{(j)}{c_{n}}\bigr|$ we have to square both sides of equalities \eqref{c_j_formula_1} and \eqref{c_j_formula_2} and then sum them up: $$\bigl|\stackrel{(j)}{c_{n}}\bigr|\leq\Bigl(\stackrel{(0)}{\lambda_{n}}\Bigr)^{-\frac{1}{2}}\Biggl[\Biggl(\int\limits_{0}^{1}\Bigl|\stackrel{(j)}{F_{n}}\!\!\!(\xi)\Bigr|d\xi\Biggr)^{2}+\Biggl(\int\limits_{0}^{1}\Bigl|\stackrel{(j)}{F_{n}}\!\!\!(\xi)\Bigr|d\xi+\bigl(\beta/\sqrt{\stackrel{(0)}{\lambda_{n}}}\bigr)\int\limits_{0}^{\alpha}\Bigl|\stackrel{(j)}{F_{n}}\!\!\!(\xi)\Bigr|d\xi\Biggr)^{2}\Biggr]^{1/2}\leq$$
\begin{equation}\label{c_estim}
    \leq \frac{\sqrt{1+(1+\beta/\sqrt{\stackrel{(0)}{\lambda_{n}}})^2}}{\sqrt{\stackrel{(0)}{\lambda_{n}}}}\int\limits_{0}^{1}\Bigl|\stackrel{(j)}{F_{n}}\!\!\!(\xi)\Bigr|d\xi=\frac{\sqrt{1+(1+\beta/\sqrt{\stackrel{(0)}{\lambda_{n}}})^2}}{\sqrt{\stackrel{(0)}{\lambda_{n}}}}\Bigl\|\stackrel{(0)}{F_{n}}\!\!\!(x)\Bigr\|_{0,1}.
\end{equation}
Using estimate \eqref{c_estim} and formulas \eqref{u_j_left_representation}, \eqref{u_j_right_representation} we can easily estimate the value of $\bigl\|\stackrel{(j)}{u_{n}}\bigr\|_{\infty}=\max\limits_{x\in [0,1]}\bigl|\stackrel{(j)}{u_{n}}\!\!\!(x)\bigr|$ like the following
\begin{equation}\label{u_estim}
    \bigl\|\stackrel{(j)}{u_{n}}\bigr\|_{\infty}\leq a\bigl\|\stackrel{(0)}{F_{n}}\!\!\!(x)\bigr\|_{0,1}\leq
\end{equation}
$$\leq a \biggl[\sum\limits_{p=0}^{j-1}\bigl|\stackrel{(j-p)}{\lambda_{n}}\bigr|\bigl\|\stackrel{(p)}{u_{n}}\bigr\|_{\infty}+\left\|q\right\|_{0,1}\bigl\|\stackrel{(j-1)}{u_{n}}\bigr\|_{\infty}+A_{j-1}\Bigl(\widetilde{N}; \bigl\|\stackrel{(0)}{u_{n}}\bigr\|_{\infty},\ldots, \bigl\|\stackrel{(j-1)}{u_{n}}\bigr\|_{\infty}\Bigr)\biggr],$$
where $\|q\|_{0,1}=\int\limits_{0}^{1}|q(x)|d x$ and
\begin{equation}\label{a_n_expression}
   a=a(n)=\biggl(\sqrt{1+(1+\beta/\sqrt{\stackrel{(0)}{\lambda_{n}}})^2}+1\biggr)/\sqrt{\stackrel{(0)}{\lambda_{n}}},
\end{equation}
and $\widetilde{N}(u)=\sum\limits_{p=1}^{\infty}\left|a_{p}\right|u^{p}.$
Similarly to that we can estimate the value of $\bigl|\stackrel{(0)}{\lambda_{n}}\bigr|$ (see formulas \eqref{lambda_j}, \eqref{eq_15})
\begin{equation}\label{lambda_estim}
   \bigl|\stackrel{(j)}{\lambda_{n}}\bigr|\leq b \biggl[\sum\limits_{p=1}^{j-1}\bigl|\stackrel{(j-p)}{\lambda_{n}}\bigr|\bigl\|\stackrel{(p)}{u_{n}}\bigr\|_{\infty}+\left\|q\right\|_{0,1}\bigl\|\stackrel{(j-1)}{u_{n}}\bigr\|_{\infty}+A_{j-1}\Bigl(\widetilde{N}; \bigl\|\stackrel{(0)}{u_{n}}\bigr\|_{\infty},\ldots, \bigl\|\stackrel{(j-1)}{u_{n}}\bigr\|_{\infty}\Bigr)\biggr],
\end{equation}
where (see \eqref{c_square})
\begin{equation}\label{b_n_expression}
   b=b(n)=\max\Bigl\{1, \sqrt{\stackrel{(0)}{\lambda_{n}}}/M_{n},  \sqrt{\stackrel{(0)}{\lambda_{n}}}\sqrt{\widetilde{c}_{n}}/M_{n}\Bigr\}\geq \stackrel{(0)}{\lambda_{n}}\bigl\|\stackrel{(0)}{u_{n}}\!\!\!(x)\bigr\|_{\infty}/M_{n}.
\end{equation}
Using notation
\begin{equation}\label{v_mu_notation}
v_{j}=\frac{b}{a^{j}}\bigl\|\stackrel{{(j)}}{u_n}\bigr\|_\infty \quad  \mbox{and} \quad \mu_{j}=\bigl|\stackrel{(j)}{\lambda_n}\bigr|/a^{j-1},\quad j\in \mathbf{N}\cup\{0\}
\end{equation}
we can rewrite inequalities \eqref{u_estim}, \eqref{lambda_estim} for $j=2,3,\ldots$ in the following form:
\begin{equation}\label{43}
v_{j}\leq\sum_{p=0}^{j-1}\mu_{j-p}v_p+\|q\|_{0,1}v_{j-1}+A_{j-1}(\widetilde{N}_{1};0,v_1,...,v_{j-1}),
\end{equation}
$$\mu_{j}\leq\sum_{p=1}^{j-1}\mu_{j-p}v_p+\|q\|_{0,1}v_{j-1}+A_{j-1}(\widetilde{N}_{1};0,v_1,...,v_{j-1}),$$
where $$\widetilde{N}_{1}(v)\equiv\widetilde{N}\Bigl(\bigl\|\stackrel{(0)}{u_{n}}\bigr\|_{\infty}+v\Bigr).$$
To obtain similar estimates for $v_{1}$ and $\mu_{1}$ let us consider inequalities \eqref{u_estim}, \eqref{lambda_estim} for $j=1$ in more details. Thus,
\begin{equation}\label{D_43}
  {  \bigl\|\stackrel{(1)}{u_n}\bigr\|_\infty \leq a\biggl[\Bigl(\bigl|\stackrel{(1)}{\lambda_n}\bigr| +\left\|q\right\|_{0,1}\Bigr)\bigl\|\stackrel{(0)}{u_n}\bigr\|_\infty+\widetilde{N}\Bigl(\bigl\|\stackrel{(0)}{u_{n}}\bigr\|_{\infty}\Bigr)\biggr],}
\end{equation}
\begin{equation}\label{D_44}
   { \bigl|\stackrel{(1)}{\lambda_n}\bigr| \leq b\biggl[\|q\|_{0,1}\bigl\|\stackrel{(0)}{u_n}\bigr\|_\infty+\widetilde{N}\Bigl(\bigl\|\stackrel{(0)}{u_{n}}\bigr\|_{\infty}\Bigr)\biggr].}
\end{equation}
Using the fact that $\widetilde{N}(0)=0$ and inequality \eqref{D_43} we obtain the following estimate of $u_{1}:$
\begin{eqnarray}\label{D_45}
  v_{1} &\leq & {\left(\mu_{1} +\left\|q\right\|_{0,1}\right)v_{0}+b\widetilde{N}\Bigl(\bigl\|\stackrel{(0)}{u_{n}}\bigr\|_{\infty}\Bigr)=}\nonumber
  \\ &=& \left(\mu_{1} +\left\|q\right\|_{0,1}\right)v_{0}+b\left(\widetilde{N}\Bigl(\bigl\|\stackrel{(0)}{u_{n}}\bigr\|_{\infty}\Bigr)-\widetilde{N}\left(0\right)\right)= \\
   &=& \left(\mu_{1} +\left\|q\right\|_{0,1}\right)v_{0}+b\widetilde{N}^{\prime}\Bigl(\theta\bigl\|\stackrel{(0)}{u_{n}}\bigr\|_{\infty}\Bigr)\bigl\|\stackrel{(0)}{u_{n}}\bigr\|_{\infty}\leq \nonumber\\
   &\leq & {\left(\mu_{1} +\left\|q\right\|_{0,1}\right)v_{0}+\widetilde{N}^{\prime}\Bigl(\bigl\|\stackrel{(0)}{u_{n}}\bigr\|_{\infty}\Bigr)v_{0}=}{\left(\mu_{1} +\left\|q\right\|_{0,1}\right)v_{0}+\widetilde{N}_{1}^{\prime}(0)v_{0}.} \nonumber
\end{eqnarray}


Similarly to \eqref{D_45} we can get the following inequality for $\mu_{1}$:
\begin{equation}\label{D_46}
  { \mu_{1}\leq \left\|q\right\|_{0,1}v_{0}+\widetilde{N}_{1}^{\prime}(0)v_{0}.}
\end{equation}

Let us consider two sequences of positive real numbers $\left\{\overline{v}_{j}\right\}_{j=0}^{\infty}$ and $\left\{\overline{\mu}_{j}\right\}_{j=0}^{\infty}$ defined through the following recurrence formulas
\begin{equation}\label{D_48}
    \overline{\mu}_{1}=\left\|q\right\|_{0,1}\overline{v}_{0}+\widetilde{N}_{1}^{\prime}(0)\overline{v}_{0}=\overline{v}_{1}-\overline{\mu}_{1}\overline{v}_{0},
\end{equation}
\begin{equation}\label{D_47}
    \overline{v}_{1}=\left(\mu_{1} +\left\|q\right\|_{0,1}\right)\overline{v}_{0}+\widetilde{N}_{1}^{\prime}(0)\overline{v}_{0},
\end{equation}
\begin{equation}\label{D_49}
    \overline{\mu}_{j}=\sum_{p=1}^{j-1}\overline{\mu}_{j-p}\overline{v}_p+\|q\|_{0,1}\overline{v}_{j-1} +A_{j-1}(\widetilde{N}_{1},0,\overline{v}_1,...,\overline{v}_{j-1})= \overline{v}_{j}-\overline{\mu}_{j}\overline{v}_{0},
\end{equation}
\begin{equation}\label{44}
\overline{v}_{j}=\sum_{p=0}^{j-1}\overline{\mu}_{j-p}\overline{v}_p+\|q\|_{0,1}\overline{v}_{j-1}+ A_{j-1}(\widetilde{N}_{1},0,\overline{v}_1,...,\overline{v}_{j-1}),
\end{equation}
where (see \eqref{c_square}, \eqref{b_n_expression} and \eqref{v_mu_notation})
$$\overline{v}_0=v_0=v_0(n)=b(n)\bigl\|\stackrel{(0)}{u_{n}}\bigr\|_{\infty}\leq$$
\begin{equation}\label{v_0}
    \leq\max\Bigl\{1, \sqrt{\stackrel{(0)}{\lambda_{n}}}/M_{n},  \sqrt{\stackrel{(0)}{\lambda_{n}}}\sqrt{\widetilde{c}_{n}}/M_{n}\Bigr\} \max\Bigl\{1/\sqrt{\stackrel{(0)}{\lambda_{n}}}, \sqrt{\widetilde{c_{n}}}/\sqrt{\stackrel{(0)}{\lambda_{n}}}\Bigr\}=
\end{equation}
$$=\max\Bigl\{1,  \sqrt{\widetilde{c_{n}}}\Bigr\} \max\Bigl\{1, \sqrt{\widetilde{c_{n}}}, M_{n}/\sqrt{\stackrel{(0)}{\lambda_{n}}}\Bigr\}/M_{n}.$$
Taking into account that the inequality $\widetilde{c_{n}}\leq 1$ implies the estimate $M_{n}\leq 1,$ we get from \eqref{v_0} the following estimate for $\overline{v}_{0}:$
\begin{equation}\label{estim_for_v_0}
   \overline{v}_{0}\leq \max\Bigl\{1, \widetilde{c_{n}}, \sqrt{\widetilde{c_{n}}}M_{n}/\sqrt{\stackrel{(0)}{\lambda_{n}}}  \Bigr\}/M_{n}.
\end{equation}
It is easy to see that $v_{j}\leq\overline{v}_{j}$ and $\mu_{j}\leq\overline{\mu}_{j}$ $\forall j\in \mathbf{N}.$

 Eliminating $\overline{\mu}_{j}$ from system \eqref{D_48} -- \eqref{44}, we obtain recurrence formulas
 \begin{equation}\label{D_51}
    \overline{v}_{1}= \left(1+\overline{v}_{0}\right)\left(\left\|q\right\|_{0,1}\overline{v}_{0}+\widetilde{N}_{1}^{\prime} (0)\overline{v}_{0}\right),
 \end{equation}
\begin{equation}\label{D_50}
  \overline{v}_{j}=\sum_{p=1}^{j-1}\overline{v}_{j-p}\overline{v}_p+(1+\overline{v}_0)\left(\|q\|_{0,1}\overline{v}_{j-1}+A_{j-1}(\widetilde{N}_{1},0,\overline{v}_1,...,\overline{v}_{j-1})\right),\quad j\in \mathbf{N}\backslash\{1\}.
\end{equation}

Denoting by $f(z)$ the series (generating function)
\begin{equation}\label{45}
f(z)=\sum_{j=1}^\infty z^j\overline{v}_j,
\end{equation}
and using equalities \eqref{D_51}, \eqref{D_50} together with $\widetilde{N}_{1}(f(z))=\sum\limits_{j=0}^{\infty}z^{j}A_{j}(\widetilde{N}_{1}; 0, \overline{v}_{1},\ldots, \overline{v}_{j}),$ we arrive at the equation with respect to $f(z):$
\begin{equation}\label{D_52}
    f(z)=\bigl[f(z)\bigr]^{2}+z(1+\overline{v}_0)\bigl[\|q\|_{0,1}\bigl(f(z)+\overline{v}_{0}\bigr)+\widetilde{N}_{1}(f(z))+\widetilde{N}_{1}^{\prime}(0)\overline{v}_{0}-\widetilde{N}_{1}(0)\bigr].
\end{equation}

\subsection{Convergence result for the linear case ($N(u)\equiv 0$)}

Let us first consider the case when $N(u)\equiv 0.$ Equation for generating function \eqref{D_52} can be rewritten in the form

$$f(z)=\bigl(f(z)\bigr)^2+z(1+\overline{v}_0)\|q\|_{0,1}\bigl(f(z)+\overline{v}_{0}\bigr)$$
or, which is more convenience,
$$\bigl(f(z)\bigr)^2-\bigl[1-(1+\overline{v}_0)\|q\|_{0,1}z\bigr] f(z)+\overline{v}_0(1+\overline{v}_0)\|q\|_{0,1}z=0.$$
We have the quadratic equation with respect to $f(z)$ with the roots
$$f_{1,2}(z)=\frac{\left(1-\left(1+\overline{v}_{0}\right)\left\|q\right\|_{0,1}z\right)\pm \sqrt{D}}{2},$$
where
\begin{eqnarray}\label{D_1}
    D&=&(w_{1}-w_{2} z)(1/w_{1}-w_{2} z),\\
    w_{1}&=&1+2\overline{v}_{0}+2\sqrt{\overline{v}_{0}\left(1+\overline{v}_{0}\right)},\quad w_{2}=\left(1+\overline{v}_{0}\right)\left\|q\right\|_{0,1}.\nonumber
\end{eqnarray}
The solution which represents the generating function is
\begin{equation}\label{D_2}
    f\left(z\right)=\frac{1-\left(1+\overline{v}_{0}\right)\left\|q\right\|_{0,1}z- \sqrt{D}}{2}.
\end{equation}
It is obvious that the right-hand side of equality \eqref{D_2} can be expanded as a power series in $z,$ $\forall z\in \left[0, R\right]:$
\begin{equation}\label{power_series_in_z}
   f\left(z\right)=\frac{1}{2}-\frac{\left(1+\overline{v}_{0}\right)\left\|q\right\|_{0,1}}{2}z- \frac{1}{2}\biggl[\sqrt{w_{1}}-\frac{w_{2}}{2\sqrt{w_{1}}}z-\sum\limits_{p=2}^{\infty}\frac{(2p-3)!!}{(2p)!!}w_{1}^{\frac{1}{2}-p}w_{2}^{p}z^{p}\biggr]\times
\end{equation}
$$\times\biggl[\frac{1}{\sqrt{w_{1}}}-\frac{w_{2}\sqrt{w_{1}}}{2}z-\sum\limits_{p=2}^{\infty}\frac{(2p-3)!!}{(2p)!!}w_{1}^{-\frac{1}{2}+p}w_{2}^{p}z^{p}\biggr]=\frac{1}{2}-\frac{\left(1+\overline{v}_{0}\right)\left\|q\right\|_{0,1}}{2}z-$$
$$-\frac{1}{2}+\frac{1}{2}\sum\limits_{j=1}^{\infty}z^{j}w_{2}^{j}\biggl[\frac{(2j-3)!!}{(2j)!!}\Bigl(w_{1}^{j}+w_{1}^{-j}\Bigr)-\sum\limits_{p=1}^{j-1}\frac{(2p-3)!!}{(2p)!!}\frac{(2j-2p-3)!!}{(2j-2p)!!}w_{1}^{-j+2p}\biggr]=$$
$$=\frac{w_{2}\left(\left(w_{1}+w_{1}^{-1}\right)-2\right)}{4}z+$$
$$+\frac{1}{2}\sum\limits_{j=2}^{\infty}z^{j}\left(w_{2}w_{1}\right)^{j}\frac{(2j-3)!!}{(2j)!!}\biggl[\Bigl(1+w_{1}^{-2j}\Bigr)-\frac{(2j)!!}{(2j-3)!!}\sum\limits_{p=1}^{j-1}\frac{(2p-3)!!}{(2p)!!}\frac{(2j-2p-3)!!}{(2j-2p)!!}w_{1}^{-2p}\biggr]$$
where
\begin{eqnarray}\label{21}
    R=R(n)=\frac{1}{w_{1}w_{2}}=\frac{1+2\overline{v}_{0}-2\sqrt{\overline{v}_{0}\left(1+\overline{v}_{0}\right)}}{\left(1+\overline{v}_{0}\right) \left\|q\right\|_{0,1}},\quad \overline{v}_{0}=\overline{v}_{0}(n),
\end{eqnarray}
$$(2k)!!=2\cdot4\cdot\ldots\cdot 2k,\quad (2k+1)!!=1\cdot3\cdot\ldots\cdot 2k+1,\quad (-1)!!\overset{\textmd{def}}{\equiv}1.$$

Taking into account that series representation \eqref{power_series_in_z} is valid for $z=R,$ we arrive at the following inequality for coefficients of generating function \eqref{45}

$$0<R\overline{v}_{1}=\frac{\left(\left(w_{1}+w_{1}^{-1}\right)-2\right)}{4w_{1}}\leq\frac{1}{4}$$
\begin{equation}\label{coefficients_for_tailor_series}
    0<R^{j}\overline{v}_{j}=\frac{(2j-3)!!}{2(2j)!!}\biggl[\Bigl(1+w_{1}^{-2j}\Bigr)-\frac{(2j)!!}{(2j-3)!!}\sum\limits_{p=1}^{j-1}\frac{(2p-3)!!}{(2p)!!}\frac{(2j-2p-3)!!}{(2j-2p)!!}w_{1}^{-2p}\biggr]\leq
\end{equation}
$$\leq \frac{(2j-3)!!}{2(2j)!!}=\alpha_{j}, \quad j=2,3,\ldots $$

Using the Stirling's formula, we can estimate $\alpha_{j}$ like the following
$$\alpha_{j}=\frac{(2j-1)!!}{2(2j-1)(2j)!!}=\frac{(2j)!}{2(2j-1)\left((2j)!!\right)^2}=\frac{(2j)!}{2^{2j+1}(2j-1)\left(j!\right)^{2}}<$$
\begin{equation}\label{estimation_for_alpha}
    <\frac{2\sqrt{\pi j}\left(\frac{2j}{e}\right)^{2j}e^{\frac{1}{24j}}}{2^{2j+1}(2j-1)\left(\sqrt{2\pi j}\left(\frac{j}{e}\right)^{j}\right)^{2}}=\frac{e^{\frac{1}{24j}}}{2(2j-1)\sqrt{\pi j}}\leq\frac{1}{(2j-1)\sqrt{\pi j}}.
\end{equation}

Returning to notation \eqref{v_mu_notation}, we arrive at the estimates (see \eqref{b_n_expression} and \eqref{estim_for_v_0})
\begin{equation}\label{v_j_estimate}
   \bigl\|\stackrel{(j)}{u_n}\bigr\|_\infty =\frac{a^{j} v_{j}}{b}\leq\frac{a^{j}}{b}\overline{v}_{j}\leq \min\Bigl\{1, M_{n}/\sqrt{\stackrel{(0)}{\lambda_{n}}},  M_{n}/\Bigl(\sqrt{\stackrel{(0)}{\lambda_{n}}}\sqrt{\widetilde{c}_{n}}\Bigr)\Bigr\}\alpha_{j}r_{n}^{j}
\end{equation}
and
\begin{equation}\label{mu_j_estimate}
  |\stackrel{(j)}{\lambda_n}|=a^{j-1}\mu_{j}\leq a^{j-1}\overline{\mu}_{j}\leq \frac{\overline{v}_{j}a^{j-1}}{(1+\overline{v}_0)}\leq \frac{\alpha_{j} M_{n}}{\biggl(M_{n}+\max\Bigl\{1, \widetilde{c_{n}}, \sqrt{\widetilde{c_{n}}}M_{n}/\sqrt{\stackrel{(0)}{\lambda_{n}}}  \Bigr\}\biggr)R} r_{n}^{j-1},
\end{equation}
where
\begin{equation}\label{r_n_expression}
    r_{n}=\frac{1+\sqrt{1+(1+\beta/\sqrt{\stackrel{(0)}{\lambda_{n}}})^2}}{\sqrt{\stackrel{(0)}{\lambda_{n}}}R}.
\end{equation}

Inequalities \eqref{v_j_estimate}, \eqref{mu_j_estimate} imply the error estimates
$$\|u_n-\stackrel{m}{u}_n\|_\infty \leq\sum_{j=m+1}^\infty\bigl\|\stackrel{(j)}{u_n}\bigr\|_{\infty}\leq C_{n,m}\min\Bigl\{1/M_{n}, 1/\sqrt{\stackrel{(0)}{\lambda_{n}}},  1/\Bigl(\sqrt{\stackrel{(0)}{\lambda_{n}}}\sqrt{\widetilde{c}_{n}}\Bigr)\Bigr\} r_{n}^{m+1},$$
\begin{equation}\label{22}
|\lambda_n-\stackrel{m}{\lambda_n}|\leq\sum_{j=m+1}^\infty \bigl|\stackrel{(j)}{\lambda_n}\bigr|\leq \frac{C_{n,m}}{\biggl(M_{n}+\max\Bigl\{1, \widetilde{c_{n}}, \sqrt{\widetilde{c_{n}}}M_{n}/\sqrt{\stackrel{(0)}{\lambda_{n}}}  \Bigr\}\biggr)R} r_{n}^m
\end{equation}
provided that
\begin{equation}\label{23}
r_{n}<1,
\end{equation}
where
\begin{equation}\label{C_m_n}
C_{n,m}=\frac{\alpha_{m+1} M_{n}}{1-\frac{1+\sqrt{1+(1+\beta/\sqrt{\stackrel{(0)}{\lambda_{n}}})^2}}{\sqrt{\stackrel{(0)}{\lambda_{n}}}R}}.
\end{equation}

\begin{theorem}\label{Theorem_linear}
  The FD-method described by formulas \eqref{7}, \eqref{eq_15}, \eqref{equation_for_lambda^{(0)}}, \eqref{expression_for_c^{(0)}}, \eqref{u_j_left_representation}, \eqref{u_j_right_representation}, \eqref{lambda_j}, \eqref{c_j_formula_1}, \eqref{c_j_formula_2} converges superexponentially to the exact solution of problem \eqref{eq_1} with $N(u)\equiv 0$ provided that condition \eqref{23} holds. The error estimates for the convergent FD-method are described by formulas  \eqref{v_0}, \eqref{r_n_expression}, \eqref{22}, \eqref{C_m_n}, with $R$ determined by formula \eqref{21} and $\alpha_{j}$ determined by formula \eqref{coefficients_for_tailor_series}.
\end{theorem}
\begin{remark}\label{asymptotic_remark}
     From Theorem \ref{basic_teorem} it follows that $\stackrel{(0)}{\lambda_{n}}\geq \pi^{2}n^{2}.$ Using this fact, it is not hard to verify that $M_{n}$ \eqref{M_n} tends to $1/2$ as $n\rightarrow +\infty$ and $\widetilde{c}_{n}$ \eqref{c_square} tends to $1$ as $n\rightarrow +\infty.$ Thus, we can conclude that for fixed $\alpha\in (0,1)$ and $\beta\geq 0$ there exists a positive integer $n_{0}$ such that for all $n\geq n_{0}$ condition \eqref{23} holds true.
\end{remark}

\subsection{Convergence result for the nonlinear case}
In this subsection we derive the conditions providing the convergence of the FD-method for the general case when $N(u)\not \equiv 0.$

   Equality \eqref{D_52} holds for any $z$ in the convergence interval of the power series \eqref{45} with the coefficients $\overline{v}_j$ determined via recursive formulas \eqref{v_0}, \eqref{D_51}, \eqref{D_50}. By $R$ we denote the radius of convergence of this series. We intend to prove that $R$ is nonzero (positive). For this purpose let us consider the inverse mapping $z=f^{-1}$. From \eqref{D_52} we obtain
\begin{equation}\label{D_53}
    z\left(f\right)= \frac{f-f^2}{(1+\overline{v}_0)\bigl[\|q\|_{0,1}(f+\overline{v}_{0})+\widetilde{N}_{1}(f)+\widetilde{N}^{\prime}_{1}\left(0\right)\overline{v}_{0}-\widetilde{N}_{1}\left(0\right)\bigr]}.
\end{equation}
Since $z(f)$ is holomorphic in some interval containing the point $f_{0}=0$ and  $z\left(0\right)=0,$  we can directly calculate the derivative $z^{\prime}\left(0\right)$
\begin{eqnarray}\label{D_54}
    z^{\prime}\left(0\right)&=&\lim\limits_{f\to 0}\frac{z\left(f\right)-z\left(0\right)}{f-0}=\\
    &=&\lim\limits_{f\to 0}\frac{1-f}{(1+\overline{v}_0)\bigl[\|q\|_{0,1}(f+\overline{v}_{0})+\widetilde{N}_{1}(f)+\widetilde{N}_{1}^{\prime}\left(0\right)\overline{v}_{0}-\widetilde{N}_{1}\left(0\right)\bigr]}=\nonumber\\
    &=&\frac{1}{\overline{v}_{0}(1+\overline{v}_0)\bigl[\|q\|_{0,1}+\widetilde{N}_{1}^{\prime}\left(0\right)\bigr]}>0.
    \nonumber
\end{eqnarray}

 Inequality \eqref{D_54} implies that there exists an inverse function $f\left(z\right)$ which is holomorphic in some interval $(-R,R)$ \cite[p. 87]{Analytic_theory}). Now let us prove that series \eqref{45} also converges at the endpoints of the interval. For this purpose it is enough to consider only right endpoint $z=R.$ Conversely, suppose that series \eqref{45} diverges at the point $z=R,$ that is,
 $$\lim\limits_{z\to R-0}f(z)=+\infty.$$
 However, taking into account that equality \eqref{D_52} holds for all $z$ in $(-R, R)$, we get the following contradiction
 \begin{eqnarray}
   1 = \lim\limits_{z\to R-0}\Biggl(f(z)+\Biggr. \\
     \left. + \frac{(1+\overline{v}_0)z\bigl[\|q\|_{0,1}(f(z)+\overline{v}_{0})+\widetilde{N}_{1}(f(z))+\widetilde{N}_{1}^{\prime}\left(0\right) \overline{v}_{0}-\widetilde{N}_{1}\left(0\right)\bigr]}{f(z)}\right) = +\infty.
 \end{eqnarray}
   This contradiction implies the inequality $f(R)<+\infty$. Thus, for some positive constants $c$ and $\varepsilon$ the following inequality holds
$$R^{j}\overline{v}_j\leq\frac{c}{j^{1+\varepsilon}},$$
where the constant $R=R(n)$ is uniquely determined by $\|q\|_{0,1},$ $\overline{v}_{0}=\overline{v}_{0}(n)$ and $\widetilde{N}_{1}\!\!\left(u\right).$

Returning to notation \eqref{v_mu_notation} we arrive at inequalities \eqref{v_j_estimate}, \eqref{mu_j_estimate} and error estimates \eqref{22}, as it was in the linear case. However, in the nonlinear case, parameter $R$ denotes the convergence radius to the power series \eqref{45} satisfying equality \eqref{D_52} and $\alpha_{j}=c/j^{1+\varepsilon}.$  Remark \ref{asymptotic_remark} is consistent with the nonlinear case too. Taking it into account we arrive at the following theorem.

\RestyleAlgo{ruled}
\DontPrintSemicolon

\begin{theorem}\label{Theorem_nonlin}
  For any $\alpha\in(0,1),$ $\beta\geq 0,$ $q(x)\in L_{1}(0,1)$ and $N(u)\in C^{\infty}(\mathbf{R})$ there exists a positive integer $n_{0},$ such that  $\forall n\geq n_{0}$ condition \eqref{23} holds and
  the FD-method described by formulas \eqref{7}, \eqref{eq_15}, \eqref{equation_for_lambda^{(0)}}, \eqref{expression_for_c^{(0)}}, \eqref{u_j_left_representation}, \eqref{u_j_right_representation}, \eqref{lambda_j}, \eqref{c_j_formula_1}, \eqref{c_j_formula_2} converges superexponentially to the exact solution of problem \eqref{eq_1}. The error estimates for the convergent FD-method are described by formulas \eqref{22}, \eqref{r_n_expression}, \eqref{C_m_n}, \eqref{v_0}, where $R$ denotes the convergence radius to the power series \eqref{45} satisfying equality \eqref{D_52} and $\alpha_{j}=c/j^{1+\varepsilon}$ for some positive constants $c$ and $\varepsilon.$
\end{theorem}

\section{Algorithm of the FD-method: software implementation}\label{s_5}

In this section we discuss the question of possible software implementation for the proposed FD-method.

Generally speaking, to apply FD-method to problem \eqref{eq_1} we should perform two main steps: the first one is to solve the basic problem \eqref{eq_12} -- \eqref{eq_14} and the second one is to calculate a certain number of corrections $\stackrel{(j)}{\lambda_{n}},$ $\stackrel{(j)}{u_{n}}\!\!(x),$ sufficient to achieve the required accuracy. The second step includes the repeated application of formulas \eqref{u_j_left_representation}, \eqref{u_j_right_representation}, \eqref{lambda_j}, \eqref{c_j_formula_1},\eqref{c_j_formula_2}, which contain the operator of integration. Hence, the second step can be executed analytically (using analytical integration) only for the rare special cases, and in practice we almost always have to use numerical integration instead. Using numerical integration, however, we approximate eigenfunction on the quadrature mesh only, whereas analytical integration yields us the uniform approximation on the interval $(0,1).$ In the algorithm described below we approximate the integral operators using Sinc quadratures and Stenger's formula (see \cite{Stenger_new}):

\begin{equation}\label{Stenger's_formula}
   \int_{a}^{z_{k}}f(x)d x\approx h_{s}\sum\limits_{p=-K}^{K}\delta_{k-p}f\left(\frac{a+be^{p h_{s}}}{1+e^{p h_{s}}}\right)\frac{\left(b-a\right)}{\left(e^{-p h_{s}/2}+e^{p h_{s}/2}\right)^{2}},\footnote{ The function $f(x)$ is needs to be sufficiently smooth on (a, b), see \cite{Stenger_new}. }
\end{equation}
where
$$z_{k}=\frac{a+be^{h_{s}k}}{1+e^{h_{s}k}},\;k=-K \ldots, K,\quad \delta_{k}=\frac{1}{2}+\int\limits_{0}^{k}\frac{\sin\left(\pi t\right)}{\pi t}d t, \; k=-2K\ldots 2K,$$
and $h_{s}=\sqrt{\frac{\pi d}{\mu K}}$ for some $0<d<\pi,$ $\mu>0.$\footnote{For discussion on the optimal choice of value for $h_{s},$ and parameters $d$ and $\mu$ see \cite{Stenger_new}, \cite{Okayama}.}

Also, we make the assumption that function $q(x)$ is analytical on the intervals $(0, \alpha),$ $(\alpha, 1)$ and might has the integrable singularities at the points $0,$ $\alpha,$ $1$ only. However, the proposed algorithm can be easily generalized on the case when $q(x)$ has any finite number of the integrable singularities on $(0,1).$ In the reminder part of the section we will use the notation
\begin{equation}\label{alg_implement_12}
    z_{1,j}=\frac{\alpha e^{h_{s}j}}{1+e^{h_{s}j}},\quad
    z_{2,j}=\frac{\alpha+e^{h_{s}j}}{1+e^{h_{s}j}}, \end{equation}
$$    \mu_{1,j}=\frac{\alpha}{\left(e^{-jh_{s}/2}+e^{jh_{s}/2}\right)^{2}},\quad
\mu_{2,j}=\frac{1-\alpha}{\left(e^{-jh_{s}/2}+e^{jh_{s}/2}\right)^{2}},\quad
$$
$$\nu_{n}=\left[\int\limits_{-1}^{1}\left(u_{n}^{(0)}(x)\right)^{2}d x\right]^{-1},\quad j=-K,\ldots, K.$$
Evidently, the highest FD-method's accuracy which can be achieved using Algorithm \ref{Algo_1} presented below is restricted by two factors: the accuracy of quadrature formulas and the accuracy with which the basic problem can be solved\footnote{Here we intentionally do not consider yet another factor: the machine precision. We assume that it is high enough (arbitrary). The arbitrary precision in the computational software can be provided via the libraries for multiple-precision floating-point computations, like MPFR \cite{MPFR}.}. To emphasize this fact, we will denote by $\Lambda_{n}^{(p)}$ the value of $\stackrel{(p)}{\lambda_{n}}$ perturbed with this two factors. For the same reasons we will use the notation $U^{n,(p)}_{i,j}$ instead of $\stackrel{(p)}{u_{n}}\!\!(z_{i,j}),$  notation $\Lambda_{n}^{p}$ instead of $\stackrel{p}{\lambda_{n}}$ and notation $U^{n,p}_{i,j}$ instead of $\stackrel{p}{u_{n}}\!\!(z_{i,j}).$

\begin{algorithm}[H]\label{Algo_1}
\caption{FD-method's algorithm for solving problem \eqref{eq_1}}
\SetKwInOut{Input}{input}\SetKwInOut{Output}{output}
\Input{The real numbers $\alpha\in (0,1),$ $\beta>0,$ $\varepsilon>0,$ functions $q(x)\in L_{1}(0,1),$ $N(u)\in C^{\infty}(\mathbf{R}),$ $N(0)=0;$\\
the number $n>0$ of eigensolution to be approximated;\\ the rank $r$ of the FD-method;}
\Output{The value $\Lambda_{n}^{r}$ such that $\bigl|\overset{r}{{\lambda}_{n}}-\Lambda_{n}^{r}\bigl|\leq \varepsilon$ and the array of values $U_{i,j}^{n,r},$ such that  $\bigl|\stackrel{r}{u_{n}}(z_{i,j})-U_{i,j}^{n,r}\bigr|\leq \varepsilon,$ $i=1,2,$ $j=-K.\ldots, K;$  }
\Begin{
 \nl  {\bf SolveBasicProblem($n$);}\tcp*{Solve the basic problem}
 \nl  {\bf FindParametersForQuadratureFormulas;}\;
 \nl  \For {$k:=1$ \KwTo $r$}{
\nl     {\bf FindNextCorrectionForEigenvalue($k$);}\tcp{Find $\Lambda_{n}^{(k)}$}
  \nl     {\bf FindParameterC($k$);}\tcp{Find $c_{n}^{(k)}$}
   \nl    {\bf FindNextCorrectionForEigenfunction($k$);}\tcp{Find $U_{i,j}^{n,(k)}$}}
\nl          $\Lambda_{n}^{r}:=\sum\limits_{p=0}^{r}\Lambda_{n}^{(p)};$\;
\nl          \For {$j:=-K$ \KwTo $K$}{
\nl               $U_{1,j}^{n,r}:=\sum\limits_{p=0}^{r}U_{1,j}^{n,(p)};\;$
              $U_{2,j}^{n,r}:=\sum\limits_{p=0}^{r}U_{2,j}^{n,(p)};$
          }
}

\end{algorithm}

Let us consider Algorithm \ref{Algo_1} in more detail.

 Procedure {\bf SolveBasicProblem($n$)} calculates the approximation for solution to the basic problem \eqref{eq_12} -- \eqref{eq_14}. If $n\alpha\in \mathbf{N}$ then $\stackrel{(0)}{\lambda_{n}}=\Lambda_{n}^{(0)}=\pi^{2} n^{2}$ and $\stackrel{(0)}{c_{n}}=C_{n}^{(0)}=-\cos(\pi n)/\sqrt{\Lambda_{n}^{(0)}}=(-1)^{n+1}/\sqrt{\Lambda_{n}^{(0)}};$ if not, then the approximation $\Lambda_{n}^{(0)}$ for $\stackrel{(0)}{\lambda_{n}}$ can be found from equation \eqref{equation_for_lambda^{(0)}} via Newton's method and approximation $C_{n}^{(0)}$ for the parameter $\stackrel{(0)}{c_{n}}$ can be found using formula \eqref{expression_for_c^{(0)}}.  After that, the approximation $U_{i,j}^{n,(0)}$ for $\stackrel{(0)}{u_{n}}\!\!\!(z_{i,j})$ can be calculated using formulas \eqref{eq_15}, \eqref{expression_for_c^{(0)}} as it is described in the procedure.

 Procedure {\bf FindParametersForQuadratureFormulas} calculates the values of parameters $h_{s}$ and $K,$ such that when quadrature formula \eqref{Stenger's_formula} is applied to integrals \eqref{u_j_left_representation}, \eqref{u_j_right_representation}, \eqref{lambda_j}, \eqref{c_j_formula_1},\eqref{c_j_formula_2} it provides the required accuracy $\varepsilon.$ Actually, the parameters $h_{s}$ and $K$ can be found through the a posteriori error analysis of quadrature formula \eqref{Stenger's_formula} applied to the integral $\int\limits_{0}^{1}\Bigl(q(x)+\stackrel{(0)}{u_{n}}\!\!\!(x)+N\Bigl(\stackrel{(0)}{u_{n}}\!\!\!(x)\Bigr)\Bigr)d x$ (see \cite{Stenger_new}).

\begin{procedure}[H]
\caption{SolveBasicProblem($n$)}
\SetKwInOut{Input}{input}
\SetKwInOut{Output}{output}
\Input{The number $n>0$ of eigensolution  to be approximated;}
\Output{The value $\Lambda^{(0)}_{n}$ such that $\bigl|\overset{(0)}{{\lambda}_{n}}-\Lambda^{(0)}_{n}\bigl|\leq \varepsilon$ and the array of values $U_{i,j}^{n,(0)},$ such that  $\bigl|\stackrel{(0)}{u_{n}}(z_{i,j})-U_{i,j}^{n,(0)}\bigr|\leq \varepsilon,$ $i=1,2,$ $j=-K\ldots, K;$  }
\Begin{
\nl \eIf{$n\alpha \in \mathbf{N}$}{
\nl$\Lambda^{(0)}_{n}:=\pi n;\;$$C_{n}^{(0)}:=(-1)^{n+1}/\Lambda_{n}^{(0)};$
}
{
\nl$\Lambda_{n}^{(0)}:=\pi n;\;$ $\Lambda:=0;\;$ \tcp*{Using the Newton's method}
\nl\While {$|\Lambda_{n}^{(0)}-\Lambda|>\varepsilon$}{
\nl   $\Lambda:=\Lambda_{n}^{(0)};\;$
  $\Lambda_{n}^{(0)}:=\Lambda-\frac{f(\Lambda)}{f^{\prime}(\Lambda)};$
  \tcp*{ Here $f(\lambda)=\frac{\sin(\lambda)}{\sin(\lambda\alpha)\sin(\lambda(1-\alpha))}+\frac{\beta}{\lambda}$}
}
\nl$C_{n}^{(0)}:=\frac{\sin(\Lambda_{n}^{(0)}\alpha)}{\Lambda_{n}^{(0)}\sin(\Lambda_{n}^{(0)}(1-\alpha))};$
}
\tcp{Compute approximations for $\stackrel{(0)}{u_{n}}(z_{i,j}),$ using formulas \eqref{eq_15}}
\nl     \For {$j:=-K$ \KwTo $K$}{
\nl           $U_{1,j}^{n,(0)}:=\sin(\Lambda_{n}^{(0)}z_{1,j})/\Lambda_{n}^{(0)};\;$
$U_{2,j}^{n,(0)}:=C_{n}^{(0)}\sin(\Lambda_{n}^{(0)}(1-z_{2,j}));$
     }
\nl     $\Lambda_{n}^{(0)}:=sqr(\Lambda_{n}^{(0)});$
}
\end{procedure}

 Procedure {\bf FindNextCorrectionForEigenvalue($k$)} calculates the approximation $\Lambda_{n}^{(k)}$ for $\stackrel{(k)}{\lambda_{n}}$ using formula \eqref{lambda_j}. The integral in the right side of this formula is approximated using {\it tanh} rule, see \cite{Kahaner_Nesh}. Also, this procedure computes the approximation $F_{i,j}$ for $\stackrel{(k)}{F_{n}}\!\!(z_{i,j})$ (see formula \eqref{n_Recurrence_sequence_of_equations}).

 Procedure {\bf ComputeParameterC($k$)} computes the approximation $C_{n}^{(k)}$ for parameter $c_{n}^{(k)}$ using formula  \eqref{c_j_formula_1} if $\alpha n\in \mathbf{N}$ or formula  \eqref{c_j_formula_2} if $\alpha n \notin \mathbf{N}.$
In the case when $\alpha n\in \mathbf{N}$ formula \eqref{c_j_formula_2} was used, though, in a slightly simplified form. Taking into account that in this case $\stackrel{(0)}{\lambda_{n}}=\pi^{2}n^{2},$ this formula was rewritten as following
\begin{equation}\label{c_simplified}
    \stackrel{(j)}{c_{n}}=(-1)^{n+1}\int\limits_{0}^{1}\frac{\cos\Bigl(\pi n\xi\Bigr)}{\pi n}\stackrel{(j)}{F_{n}}\!\!\!(\xi)d\xi+(-1)^{n}\beta
    \int\limits_{0}^{\alpha}\frac{\sin\Bigl(\pi n\xi\Bigr)}{(\pi n)^{2}}\stackrel{(j)}{F_{n}}\!\!\!(\xi)d\xi.
\end{equation}

Eventually, procedure {\bf FindNextCorrectionForEigenfunction($k$)} calculates the approximation $U_{i,j}^{n,(k)}$ for $\stackrel{ (k)}{u_{n}}\!\!(z_{i,j})$ using formulas \eqref{u_j_left_representation} and \eqref{u_j_right_representation} together with Stenger's formula \eqref{Stenger's_formula}.

\begin{procedure}[H]
\caption{FindNextCorrectionForEigenvalue($k$)}
\SetKwInOut{Input}{input}
\SetKwInOut{Output}{output}
\Input{$\Lambda_{n}^{(p)}$ for $p=0,1,\ldots, k-1;$\\
  $U_{i,j}^{n,(p)}$  for $i=1,2,$ $j=-K\ldots K,$ $p=0,\ldots,k-1;$
}
\Output{$\Lambda_{n}^{(k)};$
  $F_{i,j}\approx \stackrel{(k)}{F_{n}}\!\!\!(z_{i,j}),$ for $i=1,2,$ $j=-K\ldots K;$  }
\Begin{
\nl\For {$j:=-K$ \KwTo $K$} {
\nl    $F_{1,j}:=-\sum\limits_{p=1}^{k-1}\Lambda_{n}^{(p)}U_{1,j}^{n,(p)}+q(z_{1,j})U_{1,j}^{n,(k-1)}+A_{k-1}(N;U_{1,j}^{n,(0)},\ldots,U_{1,j}^{n,(k-1)});$\;
\nl    $F_{2,j}:=-\sum\limits_{p=1}^{k-1}\Lambda_{n}^{(p)}U_{2,j}^{n,(p)}+q(z_{2,j})U_{2,j}^{n,(k-1)}+A_{k-1}(N;U_{2,j}^{n,(0)},\ldots,U_{2,j}^{n,(k-1)});$    }
\nl  $\Lambda_{n}^{(k)}:=0;$\tcp{Compute $\stackrel{(k)}{\lambda_{n}}$ (see \eqref{lambda_j}) via the tanh rule, see \cite{Kahaner_Nesh}}
\nl  \For {$j:=-K$ \KwTo $K$} {
\nl     $\Lambda_{n}^{(k)}:=\Lambda_{n}^{(k)}+U_{1,j}^{n,(0)}F_{1,j}\mu_{1,j}+U_{2,j}^{n,(0)}F_{2,j}\mu_{2,j};$
  }
\nl  $\Lambda_{n}^{(k)}:=h_{s}\Lambda_{n}^{(k)}/\nu_{n};$\;
\nl\For {$j:=-K$ \KwTo $K$} {
\nl    $F_{1,j}:=F_{1,j}-\Lambda_{n}^{(k)}U_{1,j}^{n,(0)};\;$
    $F_{2,j}:=F_{2,j}-\Lambda_{n}^{(k)}U_{2,j}^{n,(0)};$}
}
\end{procedure}

\begin{procedure}[H]
\caption{ComputeParameterC($k$)}
\SetKwInOut{Input}{input}
\SetKwInOut{Output}{output}
\Input{$\Lambda_{n}^{(0)};\;$ $F_{i,j}$ for $i=1,2,$ $j=-K\ldots K;$}
\Output{$C_{n}^{(k)}\approx \stackrel{(k)}{c_{n}},$ see formulas \eqref{c_j_formula_1} and \eqref{c_j_formula_2};}
\Begin{
\nl $\Lambda:=\sqrt{\Lambda_{n}^{(0)}};\;$ $C^{(k)}_{n}:=0;$\;
\nl \eIf {$n\alpha \in \mathbf{N}$}{
 \tcp{\ldots then $\Lambda=\pi n$ and we use formula \eqref{c_simplified}}
\nl  \For{$j:=-K$ \KwTo $K$}{
\nl        $C^{(k)}_{n}:=C^{(k)}_{n}-\left(\cos(\Lambda z_{1,j})-\beta\sin(\Lambda z_{1,j})/\Lambda\right)F_{1,j}\mu_{1,j}-\cos(\Lambda z_{2,j})F_{2,j}\mu_{2,j};$
      }
\nl $C_{n}^{(k)}:=(-1)^{n}h_{s}C_{n}^{(k)}/\Lambda;$
 }{
 \tcp{\ldots else we use formula \eqref{c_j_formula_1}}
\nl  \For{$j:=-K$ \KwTo $K$}{
\nl        $C^{(k)}_{n}:=C^{(k)}_{n}+\sin(\Lambda(\alpha-z_{1,j}))F_{1,j}\mu_{1,j}+\sin(\Lambda(\alpha-z_{2,j}))F_{2,j}\mu_{2,j};$
      }
\nl   $C^{(k)}_{n}:=h_{s}C^{(k)}_{n}/\left(\Lambda\sin(\Lambda(1-\alpha))\right);$
 }

}
\end{procedure}
\newpage
\begin{procedure}[H]
\caption{FindNextCorrectionForEigenfunction($k$)}
\SetKwInOut{Input}{input}
\SetKwInOut{Output}{output}
\Input{$\Lambda_{n}^{(0)};\;$ $C_{n}^{(k)};\;$ $F_{i,j}$ for $i=1,2,$ $j=-K\ldots K;$}
\Output{$U_{i,k}^{n,(k)}\approx \stackrel{(k)}{u_{n}}(z_{i,j})$ for $i=1,2,$ $j=-K\ldots K,$  see  \eqref{u_j_left_representation} and \eqref{u_j_right_representation};}
\Begin{
\nl $\Lambda:=\sqrt{\Lambda^{(0)}_{n}};$\;
\nl  \For {$p:=-K$ \KwTo $K$}{
\nl        $\mathcal{I}_{1}:=0;\;$ $\mathcal{I}_{2}:=0;\;$\;
\nl       \For {$j:=-K$ \KwTo $K$} {
\nl            $\mathcal{I}_{1}:=\mathcal{I}_{1}+\cos(\Lambda z_{1,j})F_{1,j}\mu_{1,j}\delta_{p-j};\;$ \tcp{$\mathcal{I}_{1}\approx \int\limits_{0}^{z_{1,p}}\cos(\Lambda \xi)\stackrel{(k)}{F_{n}}(\xi) d\xi$}
\nl            $\mathcal{I}_{2}:=\mathcal{I}_{2}+\sin(\Lambda z_{1,j})F_{1,j}\mu_{1,j}\delta_{p-j};\;$ \tcp{$\mathcal{I}_{2}\approx \int\limits_{0}^{z_{1,p}}\sin(\Lambda \xi)\stackrel{(k)}{F_{n}}(\xi) d\xi$}
       }
\nl       $U_{1,p}^{n,(k)}:=h_{s}\Bigl(\sin(\Lambda z_{1,p})\mathcal{I}_{1}-\cos(\Lambda z_{1,p})\mathcal{I}_{2}\Bigr)/\Lambda;$\tcp{See formula \eqref{u_j_left_representation}}
  }

\nl  \For {$p:=-K$ \KwTo $K$}{
\nl        $\mathcal{I}_{1}:=0;\;$ $\mathcal{I}_{2}:=0;\;$\;
\nl       \For {$j:=-K$ \KwTo $K$} {
\nl            $\mathcal{I}_{1}:=\mathcal{I}_{1}+\cos(\Lambda z_{2,j})F_{2,j}\mu_{2,j}\delta_{j-p};\;$ \tcp{$\mathcal{I}_{1}\approx \int\limits_{z_{2,p}}^{1}\cos(\Lambda \xi)\stackrel{(k)}{F_{n}}(\xi) d\xi$}
\nl            $\mathcal{I}_{2}:=\mathcal{I}_{2}+\sin(\Lambda z_{2,j})F_{2,j}\mu_{2,j}\delta_{j-p};\;$ \tcp{$\mathcal{I}_{2}\approx \int\limits_{z_{2,p}}^{1}\sin(\Lambda \xi)\stackrel{(k)}{F_{n}}(\xi) d\xi$}
       }
\nl       $U_{2,p}^{n,(k)}:=C_{n}^{(k)}\sin(\Lambda(1-z_{2,p}))-h_{s}\Bigl(\sin(\Lambda z_{2,p})\mathcal{I}_{1}-\cos(\Lambda z_{2,p})\mathcal{I}_{2}\Bigr)/\Lambda;$\tcp{See formula \eqref{u_j_right_representation}}
  }

}
\end{procedure}

Using the proposed algorithm and the MPFR library \cite{MPFR},  a C++ application was developed for solving problems of type \eqref{eq_1}. The numerical example presented in the next section was prepared via this application.

\section{Numerical example}\label{s_6}
Let us consider the Sturm-Liouville problem \eqref{eq_1} with
\begin{equation}\label{Ex_1_problem}
  \alpha=\frac{1}{2},\quad \beta=2,\quad q(x)=\frac{1}{\sqrt{\left|0.7-x\right|}}+\frac{1}{\sqrt{\left|0.1-x\right|}}+\frac{1}{\sqrt{\left|0.3-x\right|}}+\frac{1}{\sqrt{\left|0.4-2x\right|}},
\end{equation}
$$N(u)=u^{9}.$$

For the error control we will use the following functionals

\begin{equation}\label{Error_control}
    \stackrel{m}{r_{n}}=\int\limits_{0}^{\alpha}\left|1-\frac{d \stackrel{(m)}{u_{n}}\!\!\!(\xi)}{d \xi}+\int\limits_{0}^{\xi}\left[q(\xi_{1})\stackrel{m}{u_{n}}\!\!(\xi_{1})-\lambda_{n} \stackrel{m}{u_{n}}\!\!(\xi_{1})+N(\stackrel{m}{u_{n}}\!\!(\xi_{1}))\right]d \xi_{1}\right|d\xi+
\end{equation}
$$+\int\limits_{\alpha}^{1}\left|1-\frac{d \stackrel{(m)}{u_{n}}\!\!\!(\xi)}{d \xi}+\beta \stackrel{m}{u_{n}}\!\!(\alpha)+\int\limits_{0}^{\xi}\left[q(\xi_{1})\stackrel{m}{u_{n}}\!\!(\xi_{1})-\lambda_{n} \stackrel{m}{u_{n}}\!\!(\xi_{1})+N(\stackrel{m}{u_{n}}\!\!(\xi_{1}))\right]d \xi_{1}\right|d\xi,$$
\begin{equation}\label{Jump_control}
    \stackrel{m}{\Delta}_{n}\!\!\!(\alpha)=\frac{d}{dx}\stackrel{m}{u_{n}}\!\!(\alpha+0)-\frac{d}{dx}\stackrel{m}{u_{n}}\!\!(\alpha-0)-\beta\stackrel{m}{u_{n}}\!\!(\alpha).
\end{equation}

To tackle problem \eqref{eq_1} with \eqref{Ex_1_problem} the algorithm described in section \ref{s_5} has been slightly modified. The intervals $(0, \alpha)$ and $(\alpha, 1)$ have been split  into subintervals $(0, 0.1),$ $(0.1, 0.2),$ $(0.2, 0.3),$ $(0.3, \alpha )$ $(\alpha, 0.7),$ $(0.7, 1)$ with respect to the singular points of function $q(x).$ After that all integrals over the interval $(0,1)$ were substituted by the sum of integrals over each subinterval.  The approximations for the ten eigenvalues and eigenfunctions were successfully  found and presented in Tab. \ref{Tabl_1} and Fig. \ref{pic_1} -- \ref{pic_3}. The graphs presented on Fig. \ref{pic_6}, \ref{pic_7} confirm the exponential nature of the FD-method's convergence rate.
\begin{table}[htbp]
\caption{Example 1. The results of computations.}\label{Tabl_1}
\begin{tabular}{|c|c|c|c|c|c|c|}
  \hline
  $n$ & $m$ & $\stackrel{m}{\lambda}_{n}$ & $\|\stackrel{(m)}{u_{n}}\!\!(x)\|_{\infty}$ & $|\stackrel{(m)}{\lambda_{n}}|$ & $\stackrel{m}{r_{n}}$ & $ \stackrel{m}{\Delta_{n}}\!\!(\alpha)$ \\
  \hline
   1& 10 & 23.437363200234028176652 & 1.5e-11 & 7.6e-10 & 2.5e-11 & 2.4e-26\\
  \hline
   2& 10 & 50.879953432153777724296 & 7.5e-12 & 2.4e-10 & 2.0e-11 & 5.6-27  \\
  \hline
   3& 10 & 102.294039773949565868154 & 3.0e-13 & 5.7e-11 & 1.9e-13 & 9.7e-27  \\
  \hline
   4& 10 & 167.932111361326104363494 & 2.2e-16 & 1.2e-13 & 1.1e-16 & 1.9e-27 \\
  \hline
  5 & 10 & 261.703789042290324125067 & 2.5e-17 & 6.8e-16 & 3.3e-17 & 8.3e-27 \\
  \hline
  6 & 10 & 365.290665054662412777331 & 1.5e-17 & 9.5e-16 & 3.7e-18 & 5.3e-27 \\
  \hline
   7 & 10 & 497.311217072847814939907 & 6.3e-19 & 1.2e-16 & 1.1e-19 & 7.7e-27 \\
  \hline
   8 & 10 & 642.305601325675356973240 & 6.7e-19 & 4.6e-17 & 1.3e-19 & 1.1e-26 \\
  \hline
    9 & 10 & 813.233561353244869046018 & 7.3e-21 & 1.7e-17 & 9.1e-22 & 8.5e-27 \\
  \hline
    10 & 10 & 995.761252385458344653891 & 3.7e-21 & 2.9e-18 & 4.4e-22 & 1.8e-28 \\
  \hline
\end{tabular}
\end{table}

Also, to illustrate the fact that the method's convergence rate increases along with the number of approximating eigenpair we have computed linear approximations to the functions $y_{u, n}=y_{u, n}(m)=\ln\Bigl(\bigl\|\stackrel{(m)}{u_{n}}\!\!\!(x)\bigr\|\Bigr),$
$y_{\lambda, n}=y_{\lambda, n}(m)=\ln\Bigl(\bigl|\stackrel{(m)}{\lambda_{n}}\bigr|\Bigr),$
$y_{r,n}=y_{r, n}(m)=\ln\Bigl(\stackrel{m}{r_{n}}\Bigr):$
$$\widetilde{y}_{u, n}=a_{u,n}m+b_{u,n}\approx y_{u, n}(m),$$
$$\widetilde{y}_{\lambda, n}=a_{\lambda,n}m+b_{\lambda,n}\approx y_{\lambda, n}(m),$$
$$\widetilde{y}_{r, n}=a_{r,n}m+b_{r,n}\approx y_{r, n}(m).$$
These approximations have been obtained via the method of least squares using the values of corresponding functions at the points $m=0,1,\ldots, 10.$ The parameters of linear approximations together with approximation errors  $e_{u,n}, e_{\lambda, n}, e_{r,n}$ (see \eqref{deviations}) are presented in Tab. \ref{LeastSqTab}.
$$e_{u,n}=\max\limits_{m=0,1,\ldots, 10}\bigl|a_{u,n}m+b_{u,n}- y_{u, n}(m)\bigr|,$$
\begin{equation}\label{deviations}
e_{\lambda,n}=\max\limits_{m=0,1,\ldots, 10}\bigl|a_{\lambda,n}m+b_{\lambda,n}- y_{\lambda, n}(m)\bigr|,
\end{equation}
$$e_{r,n}=\max\limits_{m=0,1,\ldots, 10}\bigl|a_{r,n}m+b_{r,n}- y_{r, n}(m)\bigr|.$$
From table \ref{LeastSqTab} it follows that the coefficients $a_{u,n},$ $a_{\lambda, n},$ $a_{r,n}$ decrease almost monotonically as parameter $n$ increases. Taking into account the comparatively small amounts of approximation errors, we can conclude that the convergence rate of the method has an exponential nature and it does increase as $n$ tends to infinity.

\begin{table}[htbp]
\caption{Example 1. The parameters of linear approximations.}\label{LeastSqTab}
\begin{tabular}{|c||c|c|c||c|c|c||c|c|c|}
  \hline
  $n$ & $a_{u,n}$ & $b_{u,n}$ & $e_{u,n}$ & $a_{\lambda, n}$ & $b_{\lambda, n}$ & $e_{\lambda, n}$ & $a_{r,n}$ & $b_{r,n}$ & $e_{r,n}$ \\
  \hline
  1 & $-2.3$ & $-1.7$ & $2.0$ & $-2.4$ & $2.8$ & $2.2$ & $-2.4$ & $-2.0$ & $2.1$ \\
  \hline
   2 & $-2.3$ & $-2.0$ & $0.6$ & $-2.3$ & $1.7$ & $5.7$ & $-2.3$ & $-2.8$ & $1.9$ \\
  \hline
   3 & $-2.6$ & $-2.4$ & $0.8$ & $-2.8$ & $3.5$ & $1.8$ & $-2.7$ & $-3.0$ & $1.3$ \\
  \hline
    4& $-3.3$ & $-2.5$ & $0.9$ & $-3.5$ & $4.1$ & $1.8$ & $-3.4$ & $-3.1$ & $0.9$ \\
  \hline
    5& $-3.5$ & $-2.4$ & $0.3$ & $-3.8$ & $4.9$ & $2.1$ & $-3.4$ & $-4.0$ & $0.8$ \\
  \hline
    6& $-3.6$ & $-2.9$ & $0.6$ & $-4.0$ & $4.7$ & $4.8$ & $-3.6$ & $-4.1$ & $0.8$ \\
  \hline
    7& $-3.9$ & $-2.7$ & $0.5$ & $-4.1$ & $5.0$ & $2.4$ & $-4.0$ & $-4.0$ & $0.4$ \\
  \hline
    8& $-3.9$ & $-3.1$ & $0.6$ & $-4.2$ & $5.5$ & $2.2$ & $-3.9$ & $-5.0$ & $1.1$ \\
  \hline
    9& $-4.3$ & $-3.3$ & $0.6$ & $-4.5$ & $5.2$ & $1.6$ & $-4.4$ & $-4.7$ & $0.9$ \\
  \hline
    10& $-4.4$ & $-3.0$ & $0.7$ & $-4.7$ & $5.2$ & $2.5$ & $-4.4$ & $-4.7$ & $1.0$ \\
  \hline
\end{tabular}
\end{table}

On the other hand, the values of parameters $\overline{v}_{0}(n),$ (see \eqref{v_0}), $R(n)$ (the convergence radius for generating function \eqref{45}, \eqref{D_52}) and $r_{n}$ (see \eqref{23}) calculated directly for the case of problem \eqref{eq_1}, \eqref{Ex_1_problem} indicate that the FD-method for eigenpairs with numbers $n=1,2,\ldots, 10$ have to be divergent ($r_{n}>1$), see Tab. \ref{Table_v_R_r}., whereas, in fact, it is convergent. This means that the convergence conditions stated in Theorems \ref{Theorem_linear}, \ref{Theorem_nonlin} are essentially overestimated.
\begin{table}[htbp]
\caption{Example 1. The values of parameters $\overline{v}_{0}(n),$ $R(n)$ and $r_{n}.$}\label{Table_v_R_r}
\begin{tabular}{|c|c|c|c||c|c|c|c|}
                   \hline
                   $n$ & $\overline{v}_{0}(n)$ & $R(n)$ & $r_{n}$ & $n$ & $\overline{v}_{0}(n)$ & $R(n)$ & $r_{n}$ \\
                   \hline
                   1 & 1.8 & 0.41e-2 & 189.9 & 6 & 2.0 & 0.34e-2 & 39.1\\
                   \hline
                   2 & 2.0 & 0.34e-2 & 125.1 & 7 & 2.0 & 0.34e-2 & 33.0\\
                   \hline
                   3 & 2.0 & 0.34e-2 & 76.5 & 8 & 2.0 & 0.34e-2 & 29.1\\
                   \hline
                   4 & 2.0 & 0.34e-2 & 59.6 & 9 & 2.0 & 0.34e-2 & 25.6\\
                   \hline
                   5 & 2.0 & 0.34e-2 & 46.3 & 10 & 2.0 & 0.34e-2 & 23.2\\
                   \hline
\end{tabular}
\end{table}

\begin{figure}[htbp]
\begin{minipage}[h]{1\linewidth}
\begin{minipage}[h]{0.48\linewidth}
\center{\rotatebox{-0}{\includegraphics[
width=1.0\linewidth]{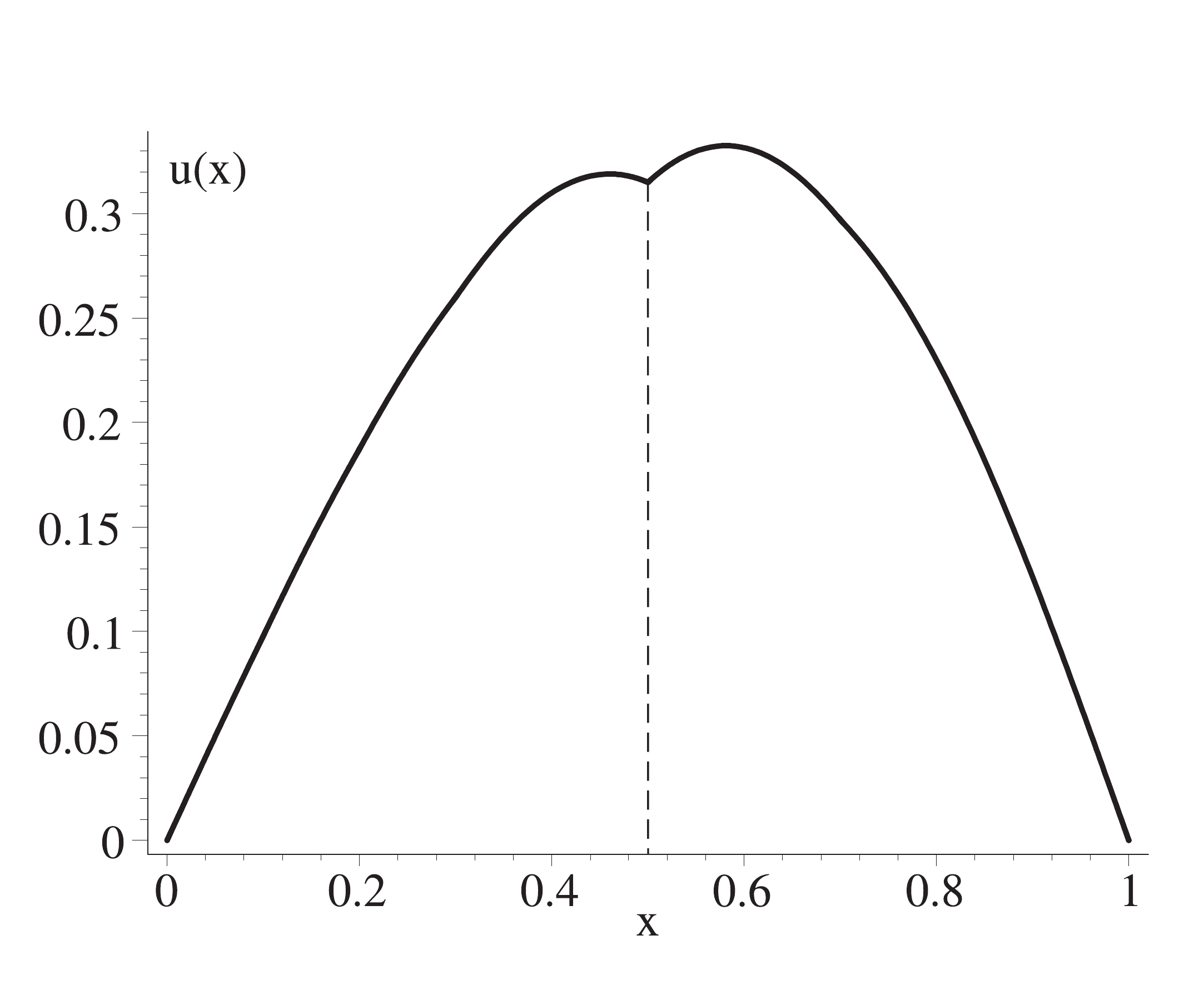}} \\ a)}
\end{minipage}
\hfill\label{image1}
\begin{minipage}[h]{0.48\linewidth}
\center{\rotatebox{-0}{\includegraphics[
width=1.0\linewidth]{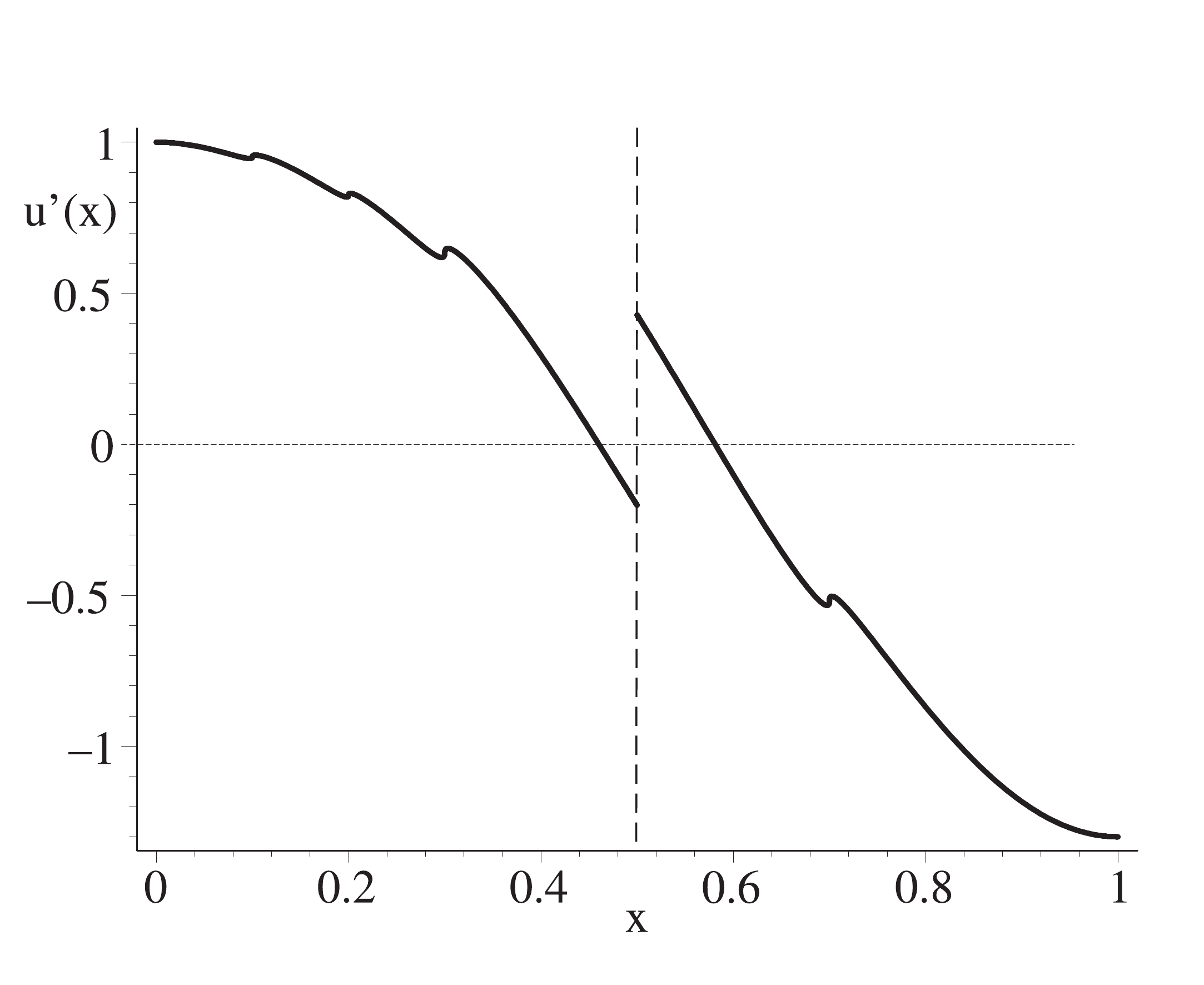}} \\ b)}
\end{minipage}
\caption{Example 1. FD-method. The graphs of
 $\stackrel{10}{u_{1}}\!\!(x)$ (figure a)) and $\frac{d}{dx}\stackrel{10}{u_{1}}\!\!(x)$ (figure b)).}\label{pic_1}
\end{minipage}
\end{figure}

\begin{figure}[htbp]\label{pic_2}
\begin{minipage}[h]{1\linewidth}
\begin{minipage}[h]{0.48\linewidth}
\center{\rotatebox{-0}{\includegraphics[
width=1.0\linewidth]{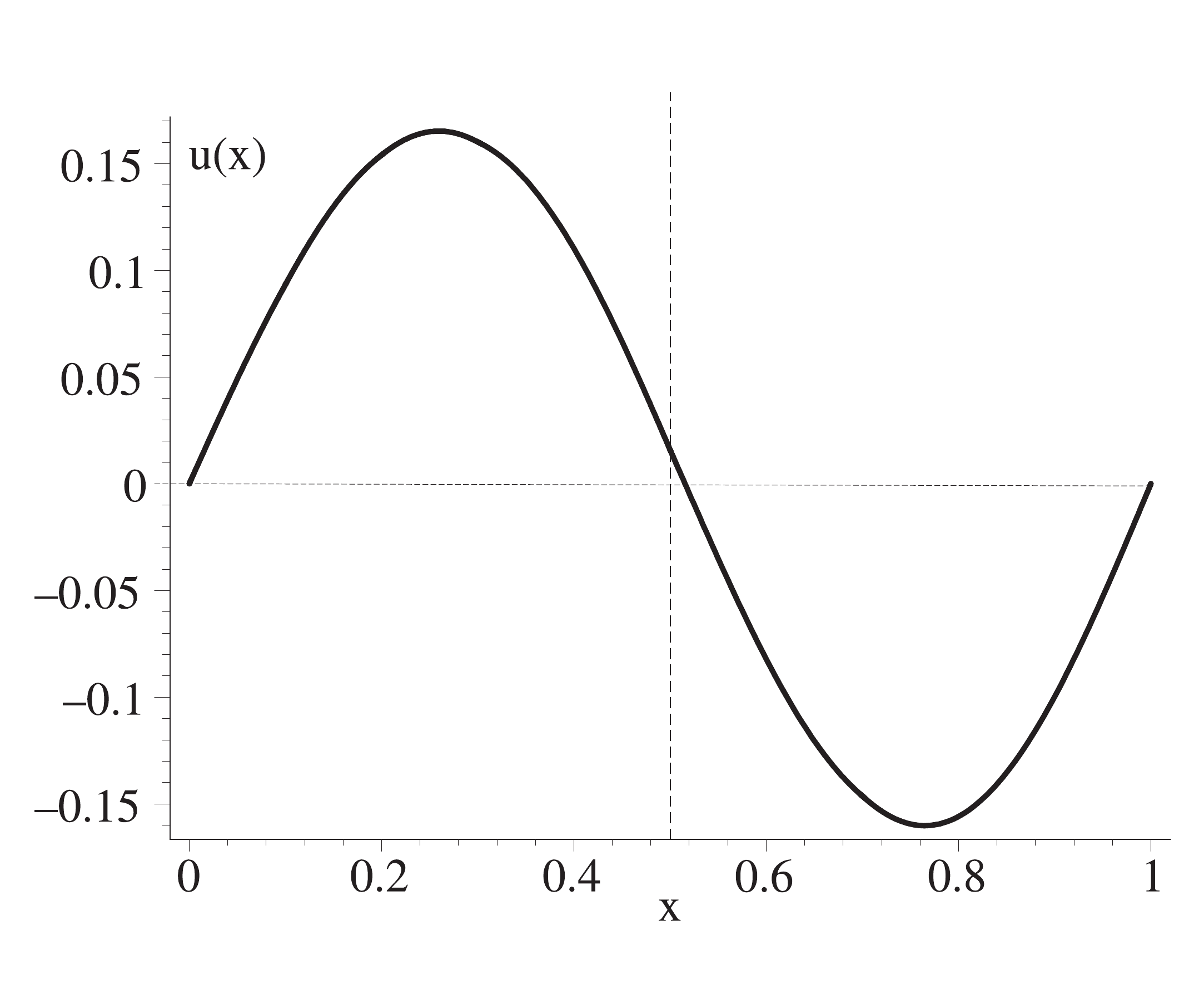}} \\ a)}
\end{minipage}
\hfill\label{image1}
\begin{minipage}[h]{0.48\linewidth}
\center{\rotatebox{-0}{\includegraphics[
width=1.0\linewidth]{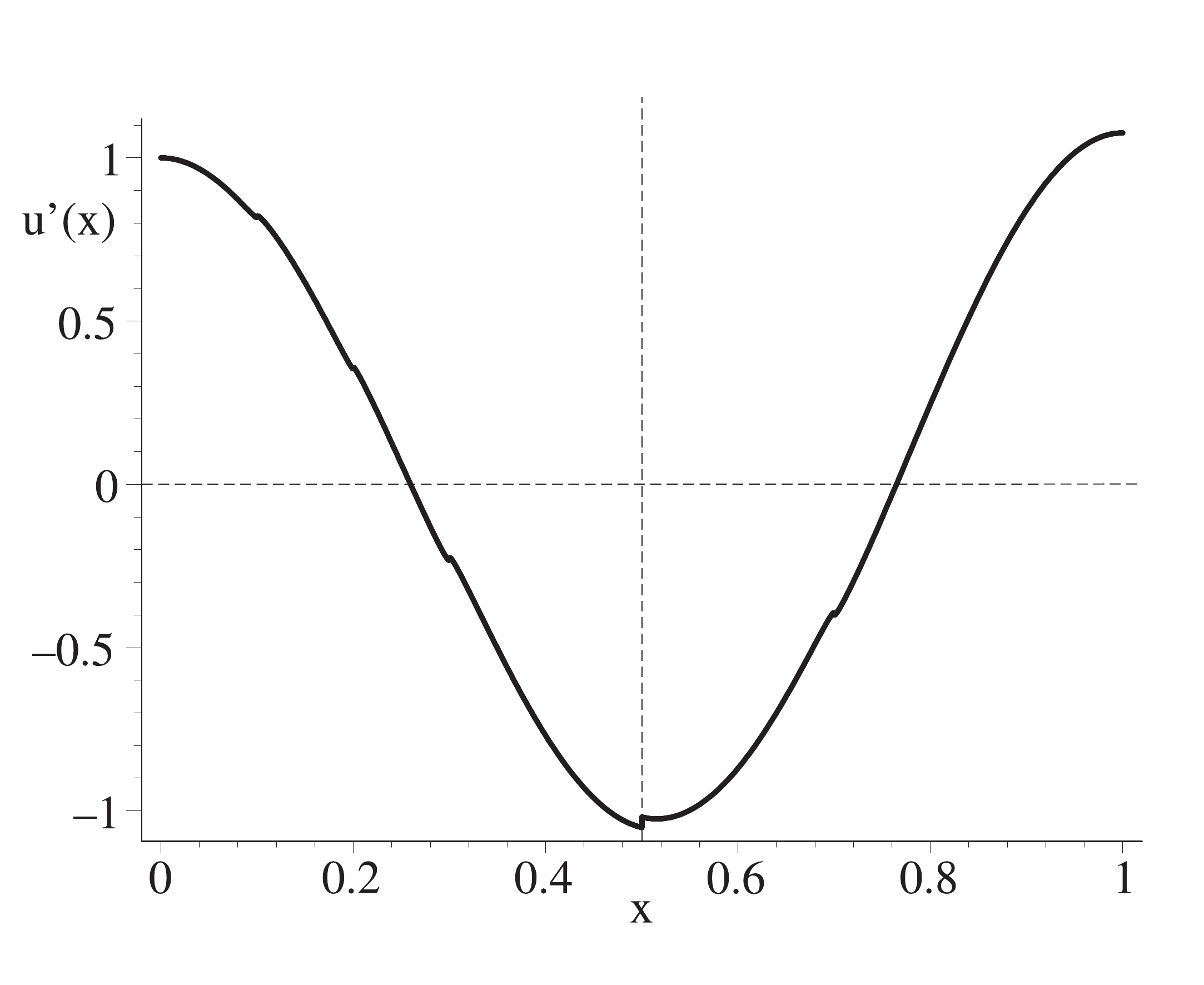}} \\ b)}
\end{minipage}
\caption{Example 1. FD-method. The graphs of
 $\stackrel{10}{u_{2}}\!\!(x)$ (figure a)) and $\frac{d}{dx}\stackrel{10}{u_{2}}\!\!(x)$ (figure b)).}\label{pic_2}
\end{minipage}
\end{figure}

\begin{figure}[htbp]\label{pic_3}
\begin{minipage}[h]{1\linewidth}
\begin{minipage}[h]{0.48\linewidth}
\center{\rotatebox{-0}{\includegraphics[
width=1.0\linewidth]{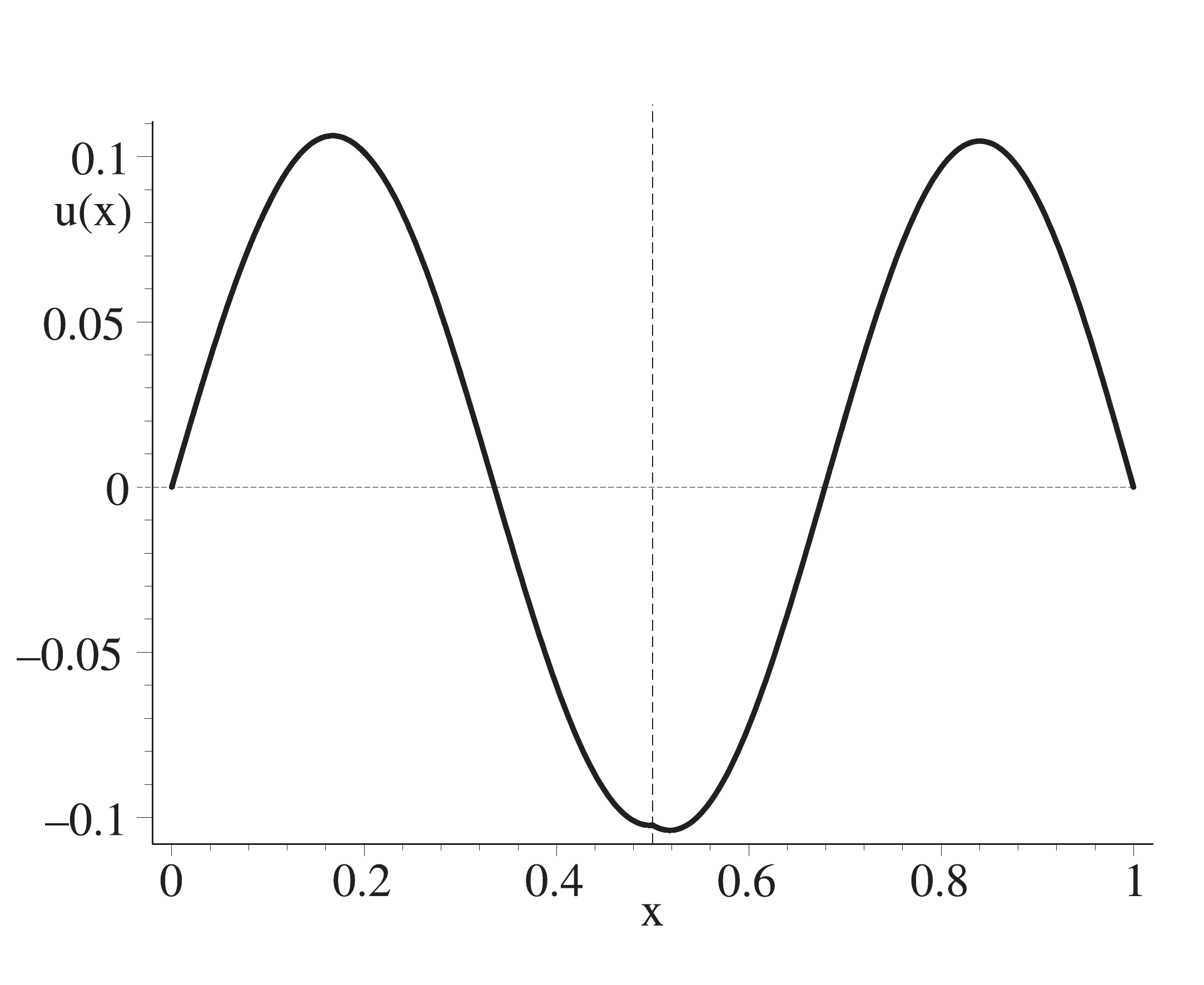}} \\ a)}
\end{minipage}
\hfill\label{image1}
\begin{minipage}[h]{0.48\linewidth}
\center{\rotatebox{-0}{\includegraphics[
width=1.0\linewidth]{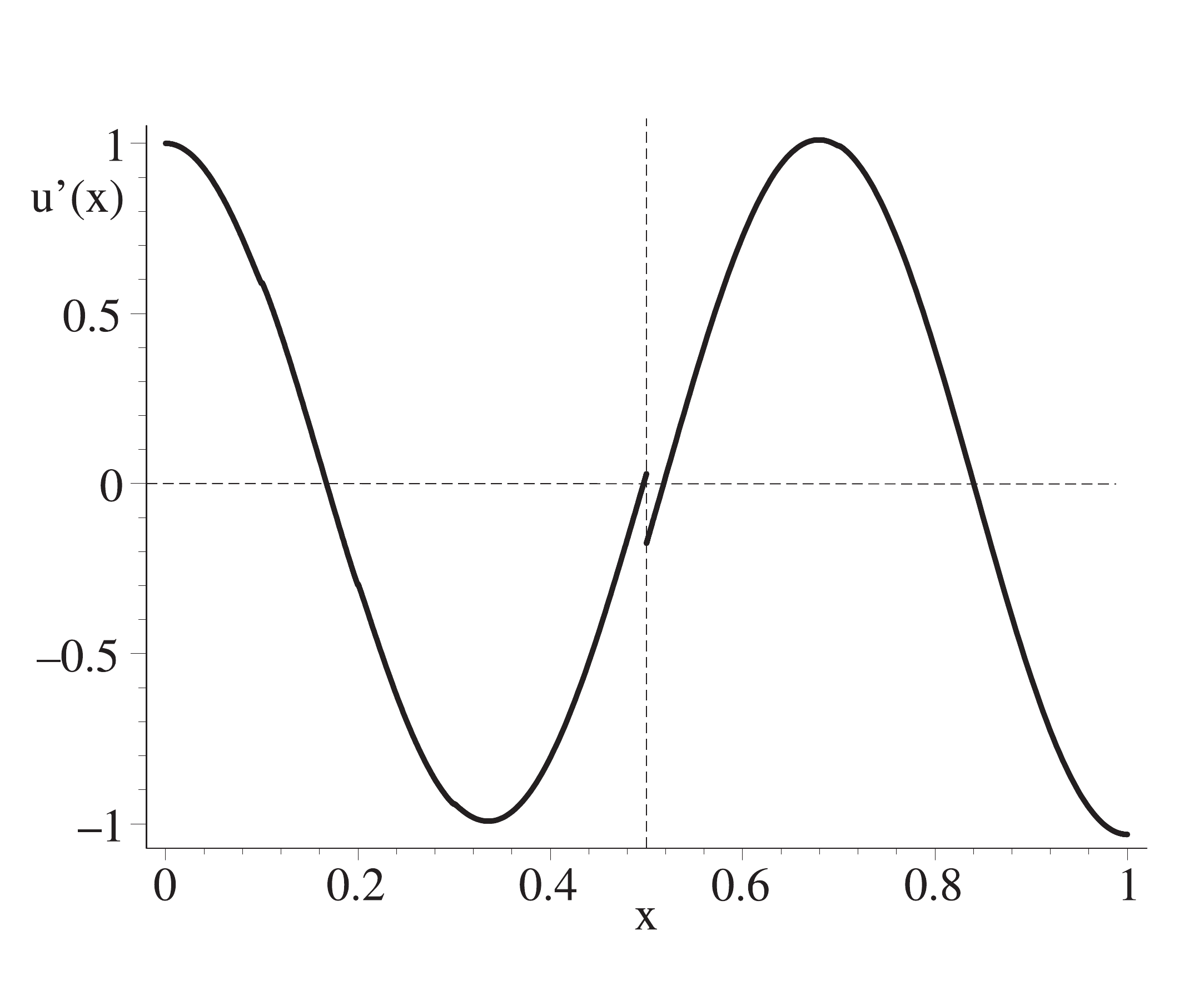}} \\ b)}
\end{minipage}
\caption{Example 1. FD-method. The graphs of
 $\stackrel{10}{u_{3}}\!\!(x)$ (figure a)) and $\frac{d}{dx}\stackrel{10}{u_{3}}\!\!(x)$ (figure b))}\label{pic_3}
\end{minipage}
\end{figure}

\begin{figure}[htbp]
\begin{minipage}[h]{1\linewidth}
\begin{minipage}[h]{0.48\linewidth}
\center{\rotatebox{-0}{\includegraphics[
width=1.0\linewidth]{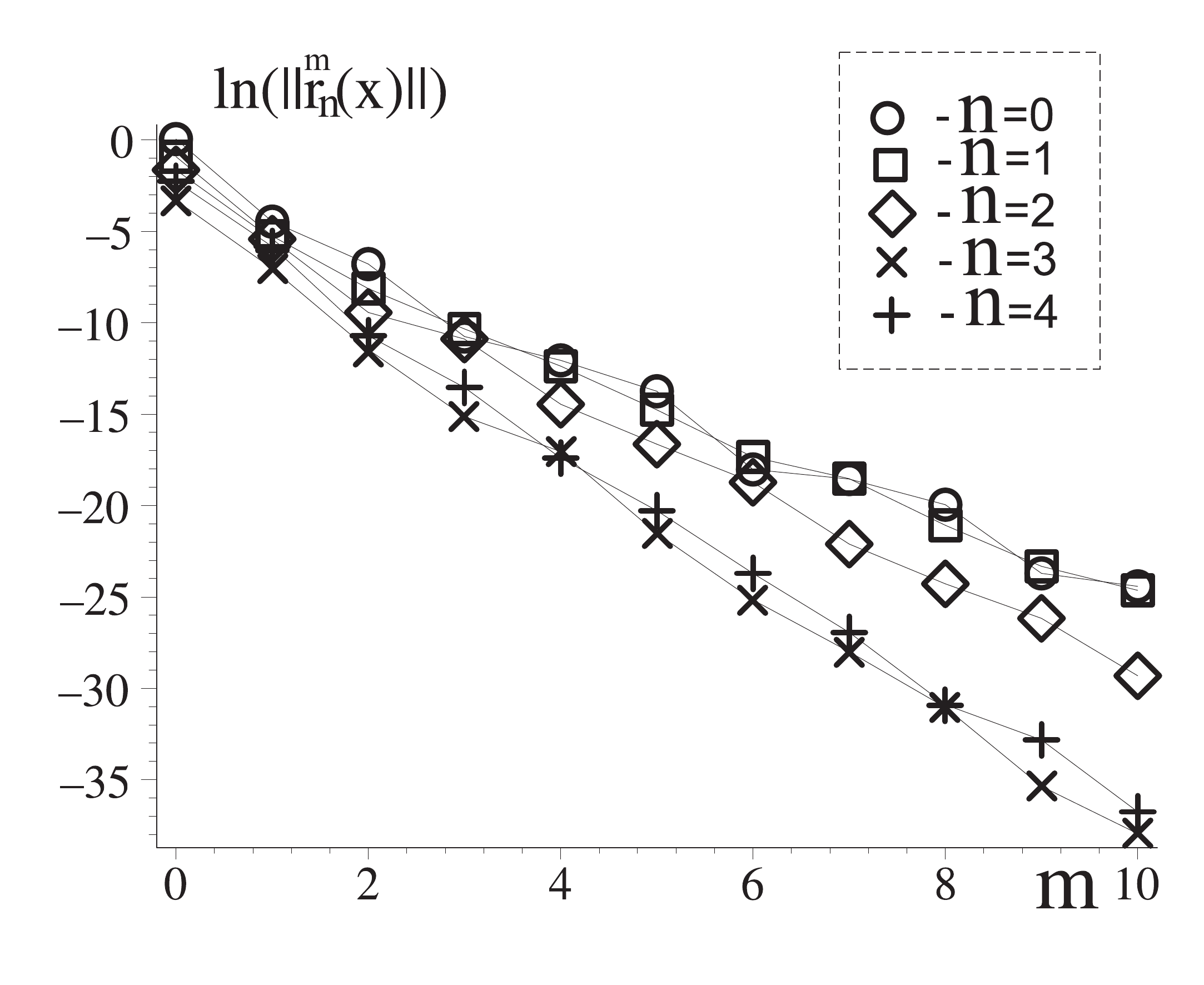}} \\ a)}
\end{minipage}
\hfill\label{image1}
\begin{minipage}[h]{0.48\linewidth}
\center{\rotatebox{-0}{\includegraphics[
width=1.0\linewidth]{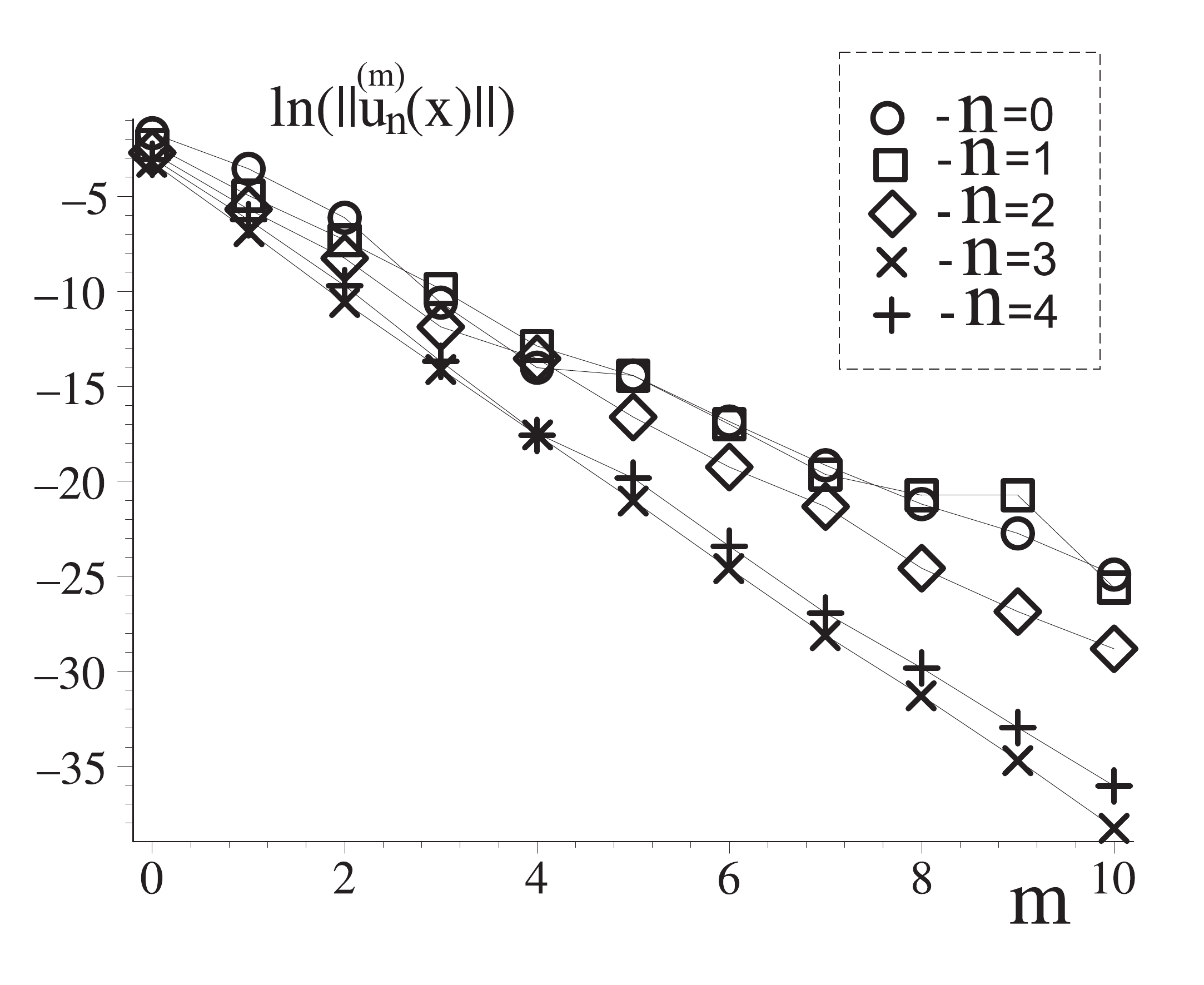}} \\ b)}
\end{minipage}
\caption{Example 1. The graphs of the functions $y_{r, n}(m)=\ln\Bigl(\stackrel{m}{r_{n}}\Bigr)$ (figure a)) and $y_{u, n}(m)=\ln\Bigl(\bigl\|\stackrel{(m)}{u_{n}}\!\!\!(x)\bigr\|\Bigr)$ (figure b)) for $n=0,1,\ldots, 4.$  }\label{pic_6}
\end{minipage}
\end{figure}

\begin{figure}[htbp]
\begin{minipage}[h]{0.5\linewidth}
\center{\rotatebox{-0}{\includegraphics[
width=1.0\linewidth]{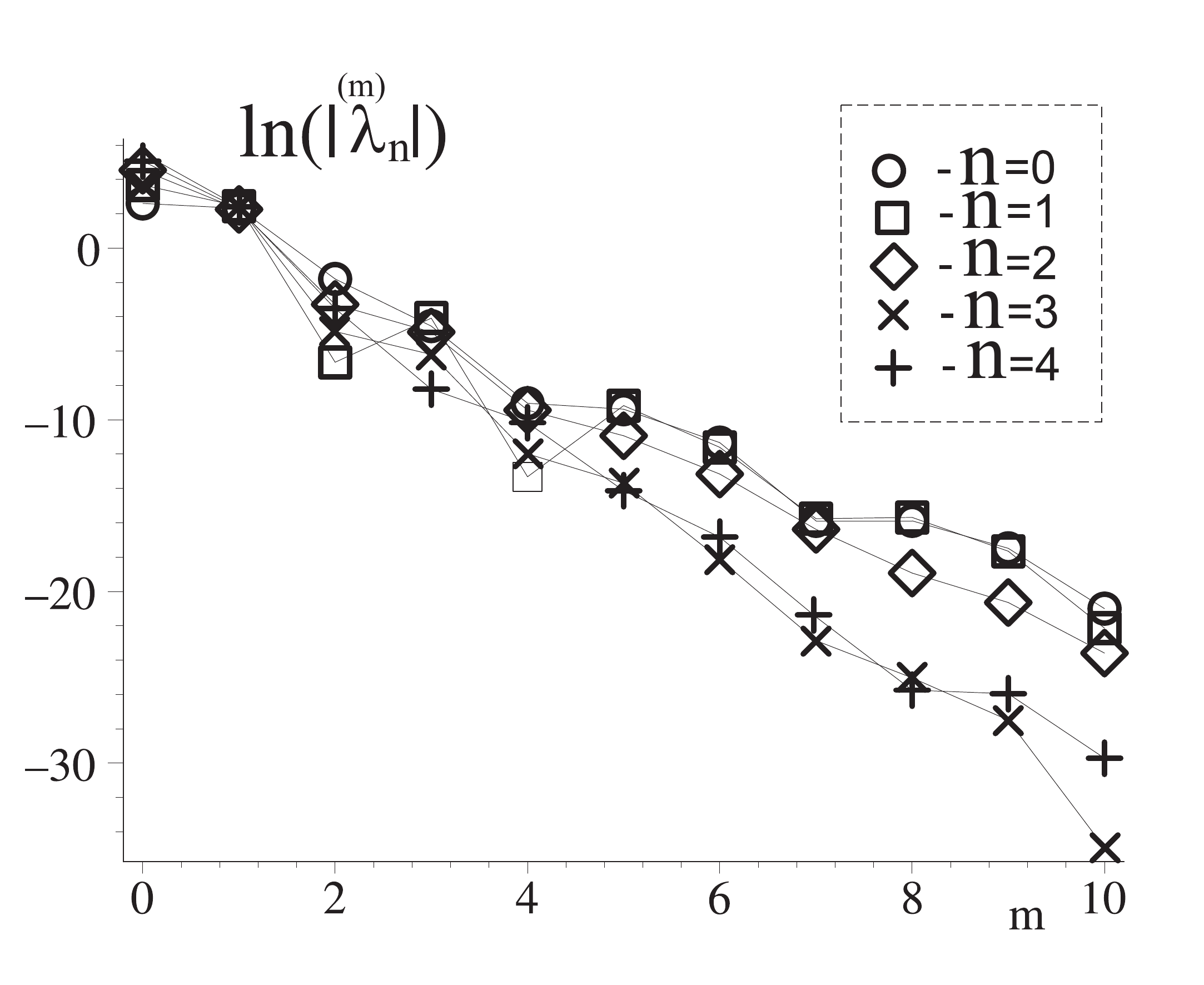}} \\ a)}
\hfill\label{image1}
\caption{Example 1. The graphs of the function $y_{\lambda, n}(m)=\ln\Bigl(\bigl|\stackrel{(m)}{\lambda_{n}}\bigr|\Bigr)$ for $n=0,1,\ldots, 4.$ }\label{pic_7}
\end{minipage}
\end{figure}

\newpage
\section{Conclusions}\label{s_7}
Summarizing the theoretical and practical results presented in the paper we can conclude that the Sturm-Liouville problems with potential including $\delta$-function can be successfully treated with the FD-approach. The authors of the paper are unaware of the software packages which can solve the problems of such type. However, the presence of $\delta$-function in the potential do not introduce significant changes to the FD-method's algorithm in comparison with that for the classic Stourm-Liouville problems. The algorithm of the FD-method can be easily modified for the case of potentials with a finite number of integrable singularities, as it was described and discussed in sections \ref{s_5} and \ref{s_6}. The numerical example presented in section \ref{s_6} confirms the predictions of Theorems \ref{Theorem_linear} and \ref{Theorem_nonlin} about the exponential nature of the FD-method's convergence.

However, as it was mentioned in section \ref{s_6}, the convergence conditions stated in theorems \ref{Theorem_linear}, \ref{Theorem_nonlin} are essentially overestimated and the search for the more subtle conditions is still a pressing issue.

\small
\bibliography{references_stat}

\def\cprime{$'$}
\begin{thebibliography}{10}

\bibitem{Albeverio_1}
S.~Albeverio, F.~Gesztesy, R.~H{\o}egh-Krohn, and H.~Holden.
\newblock {\em Solvable models in quantum mechanics}.
\newblock AMS Chelsea Publishing, Providence, RI, second edition, 2005.
\newblock With an appendix by Pavel Exner.

\bibitem{Albeverio_2}
S.~Albeverio and P.~Kurasov.
\newblock {\em Singular perturbations of differential operators}, volume 271 of
  {\em London Mathematical Society Lecture Note Series}.
\newblock Cambridge University Press, Cambridge, 2000.
\newblock Solvable Schr{\"o}dinger type operators.

\bibitem{atkinson}
F.~Atkinson.
\newblock {\em Discrete and continuous boundary problems}.
\newblock New York, 1964.

\bibitem{bglm1}
B.\u{I}. Bandyrski\u{i}, I.P. Gavrilyuk, I.I. Lazurchak, and V.L. Makarov.
\newblock Functional-discrete method ({F}{D}-method) for matrix
  {S}turm-{L}iouville problems.
\newblock {\em Computational Methods in Applied Mathematics}, 5(4):1--25, 2005.

\bibitem{Bracewell}
Ronald~N. Bracewell.
\newblock {\em The {F}ourier transform and its applications}.
\newblock McGraw-Hill Series in Electrical Engineering. Circuits and Systems.
  McGraw-Hill Book Co., New York, third edition, 1986.

\bibitem{MPFR}
Laurent Fousse, Guillaume Hanrot, Vincent Lef{\`e}vre, Patrick P{\'e}lissier,
  and Paul Zimmermann.
\newblock M{PFR}: a multiple-precision binary floating-point library with
  correct rounding.
\newblock {\em ACM Trans. Math. Software}, 33(2):Art. 13, 15, 2007.

\bibitem{ross4.1}
I.P. Gavrilyuk, A.V. Klimenko, V.L. Makarov, and N.O. Rossokhata.
\newblock Exponentially convergent algorithm for nonlinear eigenvalue problems.
\newblock {\em IMA Journal of Numerical Analysis}, 27:818--838, 2007.

\bibitem{gordon}
R.G. Gordon.
\newblock New method for constructing wavefunctions for bound states and
  scattering.
\newblock {\em Journal of Chemical Physics}, 51(1):14--25, 1969.

\bibitem{Analytic_theory}
Einar Hille.
\newblock {\em Analytic function theory}.
\newblock Ginn and Co., Boston, 1959.
\newblock 308pp.

\bibitem{Ixaru}
Liviu~Gr. Ixaru.
\newblock {\em Metode numerice pentru ecuatii diferne\c tiale cu aplica\c tii}.
\newblock Editura Academiei Republicii Socialiste Rom\^ania, Bucharest, 1979.
\newblock With an English summary.

\bibitem{seng2}
Abbaoui K., Cherruault Y., and Seng V.
\newblock Practical formulae for the calculus of multivariable adomian
  polynomials.
\newblock {\em Math. Comput. Modelling}, 22(1):89--93, 1995.

\bibitem{kryloff}
N.~Kryloff and N.~Bogolioubov.
\newblock Sopra il metodo del coefficient constati (metodo del tronconi) per
  l'integrazione approssimate delle equazioni differenziali delle fisica
  mathematica.
\newblock {\em Bolletino della Unione Mathematica Italiana}, 7(2):72--76, 1926.

\bibitem{makarov1}
V.L. Makarov.
\newblock About functional-discrete method of arbitrary accuracy order for
  solving sturm-liouville problem with piecewise smooth coefficients.
\newblock {\em Dokl. Akad. Nauk. SSSR}, 320(1):34--39, 1991.

\bibitem{makarov2}
V.L. Makarov.
\newblock About functional-discrete approach for solving problems of
  mathematical physics.
\newblock In {\em Proceedings of {W}ekua workshop}, pages 44--51, Tbilisi,
  1993.

\bibitem{makarov2008}
V.L. Makarov.
\newblock {F}{D}-method for nonlinear eigenvalue problems for nonlinear
  differential equations.
\newblock {\em Dopovidi NAN Ukrainu}, 8:16--22, 2008 (ukr.).

\bibitem{mr1}
V.L. Makarov and N.�. Rossokhata.
\newblock Error estimates of convergence rate for fd-method for sturm-liouville
  problem with potential in ${L}_1$.
\newblock {\em Collected works of Institute of Mathematics NAS of Ukraine},
  1(3):1--16, 2005 (ukr).

\bibitem{ross4.2}
V.L. Makarov and N.O. Rossokhata.
\newblock {F}{D}-method for nonlinear eigenvalue problems with discontinuous
  eigenfunctions.
\newblock {\em Nonlinear Oscillations}, 10(1):126--143, 2007.

\bibitem{mrrew}
V.L. Makarov and N.O. Rossokhata.
\newblock A rewiev of functional-discrete technique for eigenvalue problems.
\newblock {\em Journal of Numerical and Applied Mathematics}, 97:97--102, 2009.

\bibitem{mrb1}
V.L Makarov, N.O. Rossokhata, and B.\u{I}. Bandyrski\u{i}.
\newblock Functional-discrete method with a high order of accuracy for the
  eigenvalue transmission problem.
\newblock {\em Computational Methods in Applied Mathematics}, 4(3):369--381,
  2004.

\bibitem{mrb2}
V.L Makarov, N.O. Rossokhata, and B.\u{I}. Bandyrski\u{i}.
\newblock Functional-discrete method for an eigenvalue transmission problem
  with periodic boundary conditions.
\newblock {\em Computational Methods in Applied Mathematics}, 45(2):201--220,
  2005.

\bibitem{mu1}
V.L. Makarov and O.L. Ukhanev.
\newblock {F}{D}-method for {S}turm-{L}iouville problems. {E}xponential rate of
  convergence.
\newblock {\em Applied Mathematics and Informatics(Tbilisi University Press)},
  2:1--19, 1997.

\bibitem{Kahaner_Nesh}
Stephen Nash, David Kahaner, and Clev Moler.
\newblock {\em Numerical methods and software}.
\newblock Prentice-Hall, Inc., New Jersey, 1989.
\newblock 495 p.

\bibitem{pruess}
S.A. Pruess.
\newblock {\em Estimating the eigenvalues of {S}turm-{L}iouville problems by
  approximating the differential equations}.
\newblock PhD thesis, Purdue University, 1970.

\bibitem{pryce}
J.D. Pryce.
\newblock {\em Numerical solution of {S}turm-{L}iouville problems}.
\newblock Clarendon Press, Oxford, New York, Tokyo, 1993.
\newblock 322P.

\bibitem{ross3.3}
N.O. Rossokhata.
\newblock Analysis of an eigenvalue transmission problems with {FD}-method.
\newblock {\em Bulletin of the University of Kiev}, (1):194--203, 2006.

\bibitem{ross1}
N.O. Rossokhata.
\newblock {FD}-method for eigenvalue transmission problems with potential in
  space ${L}_1$.
\newblock {\em Bulletin of the University of Kiev}, 3:161--168, 2007.

\bibitem{SavShkal_7}
A.~M. Savchuk and A.~A. Shkalikov.
\newblock Sturm-{L}iouville operators with singular potentials.
\newblock {\em Mat. Zametki}, 66(6):897--912, 1999.

\bibitem{SavShkal_4}
A.~M. Savchuk and A.~A. Shkalikov.
\newblock Sturm-{L}iouville operators with distribution potentials.
\newblock {\em Tr. Mosk. Mat. Obs.}, 64:159--212, 2003.

\bibitem{Stenger_new}
Frank Stenger.
\newblock {\em Handbook of {S}inc numerical methods}.
\newblock Chapman \& Hall/CRC Numerical Analysis and Scientific Computing. CRC
  Press, Boca Raton, FL, 2011.
\newblock With 1 CD-ROM (Windows, Macintosh and UNIX).

\bibitem{Okayama}
Takayasu~MATSUO Tomoaki~OKAYAMA and Masaaki SUGIHARA.
\newblock Error estimates with explicit constants for sinc approximation, sinc
  quadrature and sinc indefinite integration.
\newblock {\em Mathematical Engineering Technical Reports 2009-01, The
  University of Tokyo, January 2009}, pages 1--28, 2009.

\bibitem{seng1}
Seng V., Abbaoui K., and Cherruault Y.
\newblock Adomian's polynomials for nonlinear operators.
\newblock {\em Math. Comput. Modelling}, 24(1):59--65, 1996.

\bibitem{vinokurov5}
V.~A. Vinokurov.
\newblock The eigenvalue and eigenfunction of the {S}turm-{L}iouville problem
  as analytic functions of the integrable potential.
\newblock {\em Differ. Uravn.}, 41(6):730--738, 861, 2005.

\bibitem{vinokurov1}
V.~A. Vinokurov and V.~A. Sadovnichi{\u\i}.
\newblock Asymptotics of arbitrary order of the eigenvalues and eigenfunctions
  of the {S}turm-{L}iouville boundary value problem in an interval with a
  summable potential.
\newblock {\em Izv. Ross. Akad. Nauk Ser. Mat.}, 64(4):47--108, 2000.

\bibitem{vinokurov2}
V.~A. Vinokurov and V.~A. Sadovnichi{\u\i}.
\newblock Asymptotics of eigenvalues and eigenfunctions and the trace formula
  for a potential that contains {$\delta$}-functions.
\newblock {\em Dokl. Akad. Nauk}, 376(4):445--448, 2001.

\bibitem{vinokurov4}
V.~A. Vinokurov and V.~A. Sadovnichi{\u\i}.
\newblock Analytic dependence of the eigenvalue and eigenfunction of the
  {S}turm-{L}iouville problem on the integrable potential.
\newblock {\em Dokl. Akad. Nauk}, 400(4):439--443, 2005.

\bibitem{vinokurov3}
V.A. Vinokurov and V.A Sadovnichii.
\newblock The asymptotics of eigenvalues and eigenfunctions and a trace formula
  for a potential with delta functions.
\newblock {\em Differential equations}, 38(6):772--789, 2002.

\bibitem{vinokurov6}
V.A. Vinokurov and V.A Sadovnichii.
\newblock Today theory of the sturm-liouville problem.
\newblock In {\em Proceedings of International Conference 'Tikhonov and
  contemporary mathematics'}, pages 383--385, Moscow, 2006.

\end{thebibliography}

\end{document}